\documentclass[11pt,a4paper]{article}
\usepackage[T1]{fontenc}
\usepackage[utf8]{inputenc}
\usepackage{amsmath,amssymb,amsfonts,amsthm,mathtools}
\usepackage[margin=3cm]{geometry}
\usepackage[colorlinks, citecolor=Green,linkcolor=blue, urlcolor = blue]{hyperref}
\usepackage{enumitem}
\usepackage[dvipsnames]{xcolor}
\usepackage{esint}
\usepackage{setspace}

\usepackage[colorinlistoftodos,prependcaption]{todonotes}

\newtheorem{trm}{Theorem}
\newtheorem{prop}[trm]{Proposition}
\newtheorem{lemma}[trm]{Lemma}

\theoremstyle{definition}
\newtheorem{defin}[trm]{Definition}
\newtheorem{rmk}[trm]{Remark}

\newcommand{\p}{\partial}
\newcommand{\eps}{\varepsilon}
\newcommand{\R}{\mathbb R}
\newcommand{\N}{\mathbb N}

\newcommand{\dV}{\,dV_g}

\makeatletter
\def\Xint#1{\mathchoice
  {\XXint\displaystyle\textstyle{#1}}%
  {\XXint\textstyle\scriptstyle{#1}}%
  {\XXint\scriptstyle\scriptscriptstyle{#1}}%
  {\XXint\scriptscriptstyle\scriptscriptstyle{#1}}%
  \!\int}
\def\XXint#1#2#3{{\setbox0=\hbox{$#1{#2#3}{\int}$}
  \vcenter{\hbox{$#2#3$}}\kern-.5\wd0}}
\makeatother
\renewcommand{\fint}{\Xint-}

\newcommand{\loc}{\mathrm{loc}}
\newcommand{\diam}{\mathrm{diam}}
\newcommand{\dist}{\operatorname{dist}}

\newcommand{\inj}{\operatorname{inj}}
\newcommand{\Vol}{\operatorname{Vol}}

\usepackage[backend=biber,style=alphabetic, doi=false, url=false, giveninits=true, sorting=nyt, maxbibnames=99, maxcitenames=99]{biblatex}
\newcommand{\Addresses}{{
  \bigskip
  \footnotesize

  \textsc{Andrea Malchiodi}\par\nopagebreak
  Scuola Normale Superiore, Piazza dei Cavalieri 7, 56126 Pisa, Italy \par\nopagebreak
  \textit{E-mail address}: \texttt{andrea.malchiodi@sns.it}

  \medskip

  \textsc{Francesco Malizia}\par\nopagebreak
  Scuola Normale Superiore, Piazza dei Cavalieri 7, 56126 Pisa, Italy \par\nopagebreak
  \textit{E-mail address}: \texttt{francesco.malizia@sns.it}

\medskip

  \textsc{Luca Martinazzi}\par\nopagebreak
  Dipartimento di Matematica, Sapienza Università di Roma, Piazzale Aldo Moro 5 - 00185 Roma, Italy \par\nopagebreak
  \textit{E-mail address}: \texttt{luca.martinazzi@uniroma1.it}

}}

\author{Andrea Malchiodi, Francesco Malizia and Luca Martinazzi}

\title{Quantization for Palais--Smale sequences of the Liouville functional on closed surfaces}

\date{}
\begin{document}
\maketitle

\begin{abstract}
    In this paper we prove a quantization property for Palais-Smale 
    sequences of functionals related to Liouville equations on compact surfaces. 
    While these equations have been extensively studied over the past decades, 
    such a quantization property had not been established so far, and has been often bypassed by the Struwe monotonicity trick introduced in \cite{struwe-Acta-1988-monotonicitytrick}. Our result, which relies on careful $L^p$ estimates to obtain gradient bounds and  
    to perform neck analysis, allows to directly apply 
    variational methods to obtain existence of solutions in non-resonant regimes. 
    We also give an example of clustering of blow-up points and a stronger assumption that, on the contrary, guarantees that the blow-up points are isolated.

\vspace{3ex}
			


\end{abstract}

\begin{spacing}{0.75}
    {\tableofcontents}
\end{spacing}

\section{Introduction and statement of the main results}

Given a closed Riemannian surface $(M,g)$ which we assume, only for simplicity, to satisfy 
\[
        \Vol_g(M)=1, 
\]
a positive function $h\in C^1(M)$ and $\lambda > 0$, we consider the 
following mean field equation of Liouville type  
    \begin{equation}\label{Liouv}
 -\Delta_g u
 =\lambda \left(\frac{h e^{u}}{\int_M h e^{u}\dV}-1\right) 
 \qquad \hbox{ on } M.
\end{equation}
Problem \eqref{Liouv} has interest both in differential geometry and in mathematical physics. 

In geometry, it is motivated by the classical Kazdan-Warner problem of prescribing the scalar curvature of a manifold via 
conformal deformations of the metric \cite{kazdanwarner-Annals-1974}, \cite{kazdanwarner-Annals-1975}. In the particular case of 
$(S^2, g_{S^2})$, the Nirenberg problem regards the prescription 
of a given (non-constant) assigned Gaussian curvature on the sphere. 

In mathematical physics, \eqref{Liouv} arises as a
mean field equation of Euler flows as well as in the description of
self-dual condensates of some Chern-Simons-Higgs model: we refer for
brevity only to \cite{cagliotilionsmarchioropulvirenti1992}, \cite{cagliotilionsmarchioropulvirenti1995}, \cite{tarantellobook} and the references therein.

\medskip

Problem \eqref{Liouv} possesses an interesting variational structure: 
its solutions can indeed be found as critical points of the functional
\begin{equation}\label{defJlambda}
J_\lambda(u):=\frac{1}{2}\int_M |\nabla_g u|^2 dV_g +\lambda\int_M u\, dV_g - \lambda\log\left(\int_M he^u dV_g   \right).
\end{equation}
A natural functional space to work with is 
\[
H^1(M) := \left\{ u : M \to \R \; : \; u, \nabla_g u \in L^2 \right\}.  
\]
In fact the  Moser-Trudinger inequality, whose sharp form can be 
found in \cite{moser-mosertrudinger} (see also \cite{fontana-mosertrudinger-manifolds}), implies that for every $u \in H^1(M)$ 
the function $e^u$ is of class $L^p$ for all $p \geq 1$ 
and that the map $u \mapsto \int_M h e^u dV_g$ is of class $C^\infty$ 
from $H^1(M)$ to $\R$, as well as for $u \mapsto J_\lambda(u)$. 

The geometric structure of $J_\lambda$ strongly depends on   
the parameter $\lambda$ and changes in particular, when $\lambda$ increases from zero to infinity, as it crosses  multiples of $8 \pi$. A consequence of 
\cite{moser-mosertrudinger} is that $J_\lambda$, up to a normalization as in \eqref{normalization}, is coercive when $\lambda < 8 \pi$: being $J_\lambda$ also weakly lower semicontinuous, by the 
direct methods of the calculus of variations one can easily find 
the existence of a minimizer. In the borderline case $\lambda = 8 \pi$ 
the minimization problem is more delicate: it has been analyzed by studying the asymptotic behavior of a family of minimizers corresponding to a 
sequence $\lambda_n \nearrow 8 \pi$, whose convergence depends on the metric $g$ and the function $h$, see e.g. \cite{djlwasian-97} and \cite{nolasco-tarantello-1998-ARMA}.

When $\lambda > 8 \pi$ it can be proven that the topology of some sub-level sets of $J_\lambda$ is non trivial, see \cite{dingjostliwang-AIHP-1999}, \cite{djadli2008allgenus}, \cite{tarantellostruwe} as well as  \cite{djadlimalchiodi-Annals-2008qcurvature} for a related geometric problem concerning $Q$-curvature, a conformal higher-order version of Gaussian curvature. This fact allows to 
employ min-max theory to try deriving existence of solutions. Min-max schemes however produce in general {\em Palais-Smale sequences}, namely sequences of approximate solutions with 
converging Euler-Lagrange energy. Then, while the Palais-Smale (PS) condition, namely pre-compactness 
of Palais-Smale sequences, would guarantee existence of true critical points, it is also known that blowing-up Palais-Smale sequences exist for $\lambda\in 8\pi \mathbb{N}\setminus\{0\}$, see e.g. \cite{AhmedouBenAyed2023}, or Theorem \ref{trmclusterexample}. 
Surprisingly, for $\lambda\not\in 8\pi \mathbb{N}\setminus\{0\},$ the validity of the Palais-Smale condition for $J_\lambda$ was not known so far, with the 
exception of the recent paper \cite{malizia-2025-ACV} by the second author, where a {\em stability} condition analogous to a  Morse index bound is assumed.
It is 
the purpose of this paper to show that the PS condition does indeed hold in general when $\lambda$ is not a positive multiple of $8 \pi$.

We point out that, however, the lack of knowledge about the PS condition has been 
circumvented using a celebrated {\em monotonicity trick} from \cite{struwe-Acta-1988-monotonicitytrick}, which allows to find {\em bounded} PS sequences, and hence solutions, for almost every value of $\lambda$. The compactness results of \cite{lishafrir-1994} and \cite{li-1999-CMP-harnacktype} lead to uniform bounds on such solutions, provided that $\lambda$ lies in a 
compact set of $\R \setminus 8 \pi \N$. With such bounds at hand, following a 
suggestion by Nirenberg, it was possible 
to compute the Leray-Schauder degree of \eqref{Liouv} in \cite{li-1999-CMP-harnacktype} and \cite{chenlin-CPAM-2003-topdegree} via the refined blow-up analysis from 
\cite{chenlin-cpam02}. 

\medskip

We can now state our main result, for which we need some preliminary facts. 
Consider a Palais--Smale sequence $(\lambda_n,\tilde u_n)\subset [0,+\infty)\times H^1(M)$ for $J_{\lambda}$ in the following sense: $\lambda_n\to\lambda_\infty\in [0,+\infty)$ and there exists $R_n\to0$ in $H^{-1}(M)$ such that
\begin{equation}\label{PSbar}
 -\Delta_g\tilde u_n
 =\lambda_n\left(\frac{he^{\tilde u_n}}{\int_M he^{\tilde u_n}\dV}-1\right)+R_n.
\end{equation}
This is equivalent to asking $J'_{\lambda_n}(\tilde{u}_n)\to 0 $ in $H^{-1}(M)$. Usually in the definition of Palais--Smale sequences one additionally requires that $J_{\lambda_n}(u_n)\to c\in \R$, but we will not use this condition.
Using the invariance of $J_\lambda$ under addition of constants, we can define
\[u_n:=\tilde u_n-\log\int_M he^{\tilde u_n}\dV,\]
which satisfies the simpler equation 
\begin{equation}\label{PSeq}
 -\Delta_g u_n=\lambda_n(he^{u_n}-1)+R_n,
        \qquad R_n\to0\quad\hbox{in }H^{-1}(M),
\end{equation}
and the normalization
\begin{equation}\label{normalization}
\int_M he^{u_n}\dV=1.
\end{equation}

\begin{defin}
We say that a sequence $(u_n)\subset H^1(M)$ of solutions to \eqref{PSeq}-\eqref{normalization} {\em blows up} if it is not precompact in $H^1(M)$. 
\end{defin} 

It is well-known that every blowing-up normalized sequence of \eqref{PSeq} exhibits a phenomenon of mass concentration, see Lemma~\ref{initialconcentration} below.
Our main result show that such concentration is quantized.

\begin{trm}\label{trmquant}
Assume that a sequence $(u_n)$ of solutions to \eqref{PSeq}, satisfying the normalization \eqref{normalization}, 
blows up. Then, there exist points $a_1,\dots, a_N\in M$ and positive integers $m_1,\dots,m_N$ such that, up to a subsequence,
\begin{equation}\label{convdelta}
\lambda_nhe^{u_n}\dV \rightharpoonup \sum_{\alpha=1}^N 8\pi m_\alpha\delta_{a_\alpha}
        \qquad\hbox{weakly as measures on }M.
\end{equation}
In particular,
\[\lambda_n\to  \lambda_\infty= 8\pi \sum_{\alpha=1}^N m_\alpha.\]
Hence, $J_\lambda$ satisfies the Palais-Smale condition 
if $\lambda$ is not a positive integer multiple of $8 \pi$.
\end{trm}

Theorem \ref{trmquant} will follow from blow-up analysis, in particular extracting \emph{bubbles} which solve the Liouville equation
\begin{equation}\label{Liouville}
-\Delta V= \mu \, e^V\quad \text{in }\R^2,\quad \mu>0,    
\end{equation}
and proving that neither mass concentrates  in the necks, nor in the weak limit. The quantum $8\pi$ of concentration then follows from the well-known classification result:

\begin{prop}[\cite{liouville1853}, \cite{chen-li-Duke-1991-classification}]\label{classification}
All solutions to \eqref{Liouville} with integrable right-hand side have the form
\begin{equation}\label{solLiouville}
V(x)=\log \left(\frac{8\delta^2}{\mu(\delta^2+|x-x_{0}|^2)^2}\right),\quad \text{for some }x_0\in \R^2,\, \delta>0,
\end{equation}
and in particular satisfy
\begin{equation}\label{liouvol}
\int_{\R^2} \mu \, e^V dx=8\pi.
\end{equation}
\end{prop}

Since the functions $u_n$ cannot a priori satisfy pointwise bounds, due to the fact that  $R_n\to 0$ in $H^{-1}(M)$, we introduce $v_n\in H^1(M)$ to be the unique solution to
\begin{equation}\label{vn}
 -\Delta_gv_n=\lambda_n(he^{u_n}-1),
        \qquad \int_M(v_n-u_n)\dV=0.
\end{equation}
Then, the bubbles are obtained from rescaling $v_n$, rather than $u_n$, and quantization requires inferring information on $e^{u_n}$ from information on $e^{v_n}$.

We will prove the following result, from which Theorem \ref{trmquant} follows at once:

\begin{trm}\label{trmmain}
Assume that a sequence $(u_n)$ of solutions to \eqref{PSeq}, {normalized as in \eqref{normalization}}, blows up. Let $(v_n)$ be as in \eqref{vn}. Then, up to a subsequence, there exist an integer $\ell\ge1$, points $x_{i,n}\in M$ with $x_{i,n}\to x^{(i)}$, radii $r_{i,n}\downarrow0$, and constants $c_{i,n}\in\R$, $i=1,\dots,\ell$, such that the following properties hold.

\begin{enumerate}[label=\textup{(\roman*)}]
\item For each $i$, in geodesic normal coordinates centered at $x_{i,n}$,
\begin{equation}\label{bubbleprofile}
 \widetilde V_{i,n}:= v_n(\exp_{x_{i,n}}(r_{i,n}\cdot ))+2\log r_{i,n}+c_{i,n}
        \longrightarrow V_i
        \quad\hbox{in }C^{1,\alpha}_{\loc}(\R^2)
\end{equation}
for every $\alpha\in(0,1)$, where $V_i$ satisfies
\begin{equation}\label{limitbubble}
 -\Delta V_i=\mu_i e^{V_i}\quad\hbox{in }\R^2,
        \qquad \int_{\R^2}\mu_i e^{V_i}\,dy=8\pi,
\end{equation}
with $\mu_i=\lambda_\infty h(x^{(i)})>0$.

\item  For every $R>0$ there exists $n_0(R)$ such that
\begin{equation}\label{scale-sep}
 B^g_{Rr_{i,n}}(x_{i,n})\cap B^g_{Rr_{j,n}}(x_{j,n})=\emptyset,
        \qquad {\forall}\,i\ne j,\quad {\forall}\,n\ge n_0(R).
\end{equation}

\item For $1\le i\le \ell$
$$\lim_{R\to\infty}\lim_{n\to\infty}\int_{B^g_{Rr_{i,n}}(x_{i,n})}\lambda_n h e^{u_n}dV_g= \lim_{R\to\infty}\lim_{n\to\infty}\int_{B^g_{Rr_{i,n}}(x_{i,n})}\lambda_n h e^{v_n+c_{i,n}}dV_g=8\pi.$$

\item $$\lim_{R\to\infty}\lim_{n\to\infty}\int_{M\setminus \bigcup_{i=1}^\ell B^g_{Rr_{i,n}}(x_{i,n})}\lambda_n h e^{u_n}dV_g =0.$$
\end{enumerate}
\end{trm}

Notice that \eqref{convdelta} follows from items (iii) and (iv) above, with
$$m_\alpha=\#\{1\le j\le \ell: x^{(j)}=a_\alpha\}.$$
In fact, the limit points $x^{(1)},\dots, x^{(\ell)}$ in Theorem \ref{trmmain} are not required to be distinct. This is not just a technical issue: differently from exact solutions, for which the Harnack-type inequality of
\cite{li-1999-CMP-harnacktype} forces the blow-up points to be at mutual distance
bounded from below, for Palais--Smale sequences clustering might occur. More precisely, we prove the following.

\begin{trm}\label{trmclusterexample}
Given $x_0\in M$ and $p\in (1,2)$, there exists a sequence $(u_n)\subset H^1(M)$ satisfying
\[
 -\Delta_gu_n=16\pi(he^{u_n}-1)+R_n \quad \text{on $M$},
 \qquad \int_{M}he^{u_n}\dV=1, 
\]
such that
$\|R_n\|_{L^p(M)}\to0$ and $(u_n)$ blows up at sequences $(x_{1,n})$ and $(x_{2,n})$ (in the notation of Theorem \ref{trmmain}) converging to the same point $x_0$. 
\end{trm}

In Theorem \ref{trmclusterexample}, $R_n\to 0$ in $L^p(M)$ for some $p\in(1,2)$ implies that $R_n\to 0$ also in $H^{-1}(M)$, so that $(u_n)$ satisfies the hypothesis of Theorems \ref{trmquant} and \ref{trmmain}. Theorem \ref{trmclusterexample} shows that even the stronger assumption $R_n\to 0$ in $L^p$ for some $p<2$, is not enough to prevent clustering at blow-up points.
We also remark that pushing $p$ to the whole range $[1,2)$ requires sharper estimates than the ones required to have $R_n\to 0$ in $H^{-1}(M)$, or even in $L^p(M)$ for $1\le p<{4/3}$. In this latter range, a more generic two-bubble ansatz suffices, while in the range $p\in [4/3,2)$ the balancing condition of Definition \ref{defbalance} will be used to produce a cancellation.

On the other hand, as already shown by \cite{ohtsukasuzuki-2003-Calcvar} in the case of a bounded domain $\Omega\subset\R^2$, with an even stronger assumption on the remainder term $R_n$, the clustering phenomenon disappears and the conclusion of \cite{li-1999-CMP-harnacktype}
is recovered:

\begin{prop}\label{propisolatedstrong}
Assume, in addition to the hypotheses of Theorem~\ref{trmmain}, that, for $w_n:=u_n-v_n$,
\begin{equation}\label{strongresidual}
        \|\nabla w_n\|_{L^\infty(M)}\leq C, \qquad \forall n\in \N,
\end{equation}
for some $C>0$.
Then the limiting centers are pairwise distinct:
\[
        x^{(i)}\ne x^{(j)}\qquad\hbox{for }i\ne j.
\]
\end{prop}

By elliptic estimates and Sobolev space embeddings, \eqref{strongresidual} is implied by
\begin{equation}\label{strongresidualbis}
        \|R_n\|_{L^p(M)}\leq C, \quad \text{for some }p>2\qquad \forall n\in \N,
\end{equation}
or by the even weaker
\begin{equation}\label{strongresidualter}
        \|R_n\|_{L^{2,1}(M)}\leq C, \qquad \forall n\in \N,
\end{equation}
where $L^{2,1}(M)\subsetneq L^2(M)$ is the Lorentz space (this uses the duality between $L^{2,1}(M)$ and $L^{2,\infty}(M)$, and the estimate in Lemma \ref{green}, which implies $\nabla_x G(x,\cdot)\in L^{2,\infty}(M)$). The borderline case, with remainder $R_n$ bounded, or even vanishing, in $L^2(M)$, remains open. Indeed, by elliptic estimates, $R_n\to 0$ in $L^2(M)$ implies that $w_n\to 0$ in $H^2(M)$, which barely fails to imply \eqref{strongresidual}.

\medskip

The proof of Proposition \ref{propisolatedstrong} is by contradiction, assuming that $x_{i,n}\to a$ for $i\in J$,  $|J|>1$, and it combines three ingredients: the gradient estimate
\eqref{gradestM}; the $L^p$-estimate \eqref{nofurther} with $p>2$, which, interpolated with
the vanishing of the $L^1$-norm of the nonlinearity away from the centers, yields
$C^{1,\alpha}$-compactness of the sequence $\hat v_n$ rescaled at the cluster scale
$\delta_n$; and a Pohozaev-type identity, which forces the limiting rescaled centers $(y_j)_{j\in J}$
to be a critical point of the logarithmic interaction energy
$\sum_{j<k}\log\frac{1}{|y_j-y_k|}$, which, on the other hand, is known to have no critical points (Lemma \ref{lem:nocriticallog}).

Let us further remark two things. The assumption of \cite{ohtsukasuzuki-2003-Calcvar} is slightly stronger than \eqref{strongresidual}, as the authors assume $\nabla w_n\to 0$ in $L^\infty$. From this, they prove the full quantization result of Theorem \ref{trmquant}, with $m_\alpha=1$. The main difficulty of our paper lies in the fact that the Palais-Smale assumption $R_n\to 0$ in $H^{-1}(M)$ only implies $w_n\to 0$ in $H^{1}(M)$, which leads to a complete loss of pointwise estimates for $w_n$ and, therefore, for $u_n$.

\medskip

With this in mind, we will next explain some of the main ideas for proving Theorem \ref{trmmain}, emphasizing the points at which the presence of the error term
$R_n$ forces arguments genuinely different from those available for exact solutions.

\medskip
\noindent\emph{Extraction of  bubbles and normalizing constants.}
A concentration--compactness argument produces
finitely many bubbling sequences $(x_{i,n},r_{i,n})$ which are \emph{simple}, i.e.
mutually separated at their own scale (see \eqref{separ}), and each of which carries exactly
$8\pi$ of conformal volume; since $\lambda_n$ is bounded, the procedure ends after finitely-many steps.
One caveat is that the bubbles are extracted from $v_n$, not
from $u_n$, and the rescaled profiles converge only after subtracting the constants
$c_{i,n}$, which are local averages of $u_n-v_n$ near $x_{i,n}$ at scales $r_{i,n}$. 
Although $\|u_n-v_n\|_{H^1(M)}\to0$ by Lemma~\ref{lemclose}, this forces only the
\emph{oscillation} of $u_n-v_n$ to vanish at the bubble scales, not its average: the
sequences $(c_{i,n})$ are in general unbounded, and no pointwise comparison between
$u_n$ and $v_n$ is available at any of the scales involved. Consequently every estimate
we prove is either normalization-free (as is \eqref{nofurther}, which involves
$e^{u_n}$ alone) or is stated for $v_n+c_{i,n}$, with a constant chosen  on each
region and  scale. The bridge between the latter two estimates is provided by
Lemmas~\ref{exppqM}--\ref{exppqmean}, which upgrade $H^1$-closeness to $L^p$ closeness of the 
exponentials (up to a shift by $c_{i,n}$).

\medskip
\noindent\emph{$L^p$-estimates at the bubble scales.}
In Proposition~\ref{propextraction} we prove an
intermediate, \emph{scale-invariant} estimate: for every fixed $p\in(2,+\infty)$,  $\rho_*>0$, there
exists $C_*>0$ such that
\[
r^2\left(\ \rlap{\bf --}\!\!\int_{B^g_r(x)}e^{pu_n}\,dV_g\right)^{1/p}\le\rho_*,
\qquad 0<r\le\frac{d_n(x)}{C_*},
\]
with $d_n$  the distance from the bubble centers. This interpolates between the case
$p=1$ of \cite{malizia-2025-ACV} and the case $p=+\infty$ of exact solutions: it is
the estimate that replaces pointwise control throughout the paper. The restriction $p>2$
is not technical: since $|\nabla_xG(x,y)|\lesssim d_g(x,y)^{-1}$, this is in $L^{p'}_{\rm loc}$ if and only if $p'<2$, so $p>2$ is
 what makes $e^{u_n}$ summable against the gradient of the Green kernel. Two
consequences follow, used repeatedly. First, the gradient estimate
$d_n(x)|\nabla_gv_n(x)|\le C$ of Proposition~\ref{propgradient},  our substitute
for the Harnack-type inequality of \cite{li-1999-CMP-harnacktype}. Second, the right-hand sides of  the
rescaled equations lie in $L^p$ with $p>2$, so the embedding $W^{2,p}\hookrightarrow C^{1,\alpha}$
provides the compactness needed in every blow-up alternative of the paper; with $L^1$
bounds only,  these arguments cannot be run.

\medskip
\noindent\emph{Cluster trees.}
Since the centers $x_{i,n}$ may converge to the same point (and they might genuinely do, by Theorem \ref{trmclusterexample}), the bubbles cannot be treated one at a time.
At each blow-up point $a_\alpha$ we organize the indices $I(\alpha)$ into a finite rooted
tree $\mathcal T(a_\alpha)$ whose nodes $A$ are clusters of bubbles, endowed with an inner
scale $s_{A,n}$ (the diameter of the cluster, the bubble scale for a leaf) and an outer
scale $S_{A,n}$ (half the distance to the closest clusters), and satisfying
$s_{A,n}=o(S_{A,n})$, see Lemma~\ref{lemclustertree}. The tree has finite height, and it
reduces to the classical picture of finitely many isolated simple bubbles precisely when
all  trees have height one. The quantization
statement is proved by induction.

\medskip
\noindent\emph{Neck analysis.}
One of the cores in the paper is Proposition~\ref{propneck}. We first test  the Palais--Smale
condition  
\eqref{PSeq} against radial logarithmic cut-offs supported in annuli
$B^g_{S_n}(z_n)\setminus B^g_{s_n}(z_n)$ centered at a cluster center (Lemma \ref{lemaverage}); since these
cut-offs are bounded in $H^1(M)$ \emph{uniformly in the two radii}, the error term
contributes only an $o(1)$, again uniformly. The resulting flux identity measures the
mean radial derivative of $u_n$ across thick annuli and yields a geometric decay of the conformal volume $\int\lambda_nhe^{u_n}$ on dyadic annuli. Two features are crucial:
the estimate is uniform in the thickness $S_n/s_n$ of the neck,  making the
dyadic summation possible; and its hypotheses are
scale-invariant, so that they can be verified both at the scale of a single bubble and at
the scale of a whole cluster. This is what makes the decay \emph{propagate} along the
tree of clusters, from the {\em leaves} to the {\em root}, see the terminology in 
Section \ref{sec:trees}.

\medskip

\noindent Proposition~\ref{propclusterquant} then proves, by induction on the height of the nodes,
that every node $A$ carries  mass $8\pi|A|$ on the ball of radius $Rs_{A,n}$ and no
mass on the surrounding neck; Proposition~\ref{nodiffuse} excludes residual mass in the
weak limit. Together, these give
$\lambda_nhe^{u_n}dV_g\rightharpoonup8\pi\sum_\alpha m_\alpha\delta_{a_\alpha}$ and
complete the proof of Theorem~\ref{trmmain}, hence of Theorem~\ref{trmquant}.
Let us also stress that we never assume $J_{\lambda_n}(u_n)$ to be bounded: only
$J'_{\lambda_n}(u_n)\to0$ in $H^{-1}$ is used.

\bigskip

It might be interesting to broaden our analysis to other situations, such as the case of singular Liouville equations, as treated in 
\cite{bartolucci-tarantello-2002-CMP} et al., 
or the above-mentioned $Q$-curvature equation, extending the result in \cite{malchiodi-2006-CRELLE}. While our approach might be flexible enough to treat  functionals with nonlinearity of the form $e^{u^p}$, $p\in (1,2)$ (for instance the ones studied in \cite{DeMarchisMalchiodiMartinazziThizy2022}), it fails in the case $p=2$. This is sharp, since Costa and Tintarev \cite{costa-tintarev-2014-JFA} 
constructed noncompact Palais--Smale sequences for the Moser-Trudinger functional violating quantization.

\medskip

We finally compare our results with two related variational problems. In \cite{Struwe1984GlobalCompactness} Struwe  studied bounded Palais-Smale sequences 
corresponding to functionals with critical growth in dimension larger than two. 
 He showed that, up to subsequences, in the limit the energy splits into that of the 
 weak limit and that of finitely many {\em bubbles}.  
Here the absence of energy loss in the necks is a direct consequence of the $H^1$ decomposition (known as \emph{Struwe decomposition}), which is based on the boundedness of Palais-Smale sequences in $W^{1,2}$. There is in general no obstruction on clustering, which can be sometimes avoided by the blow-up analysis for exact solutions  initiated by Schoen, see e.g. \cite{Druet-Y}. In our Theorem \ref{trmmain}, handling the nonlinearity is more difficult not only because our sequence $(u_n)$ is unbounded in $H^1$, but also because small perturbations in the $H^1$-norm need not have a negligible effect on the nonlinearity $e^u$ in the bubbling regions.

In the opposite direction we have the example by Parker \cite{parker-1996-JDG} related to harmonic maps from the sphere $\mathbb{S}^2$ into itself, corresponding to the conformally invariant Dirichlet energy. 
In spite of a similar scaling as our $J_\lambda$ {and a similar need to perform a neck analysis}, Palais-Smale sequences for the  Dirichlet energy  need not quantize. 
The additional amount of energy can be lost in annular necks around bubbles,  in sharp contrast with our Proposition \ref{propneck}. This is related to  geodesics, which can be seen as \emph{degenerate} harmonic maps that allow to accumulate energy inside the necks. Quantization can be recovered assuming suitable integral bounds on the {\em tension field}, see e.g. \cite{Qing-Tian}.

\paragraph{Notation} We will write
\[
\bar u:=\fint_M u\,dV_g:=\frac{1}{\Vol_g(M)}\int_M u\, dV_g,
\]
to denote the average of $H^1$ functions. We will also use the notation $\bar u_D$ to indicate the average over a specific subset $D$ of $M$ or of $\R^2$.

\medskip

The notation $x_n\to x$, $\lambda_n\to \lambda_\infty$, $o(1)\to 0$ is intended as $n\to \infty$, even if sometimes we do not explicitly write that.

\medskip

We will also use the notation $a_n\asymp b_n$ to denote positive sequences such that there exist $c,C>0$,  with
$c a_n\le b_n\le C a_n$ for any $n\in \mathbb{N}$.
\medskip
With $B_r(x)$ and $B^g_r(x)$ we will denote balls in $\R^2$ and in $(M,g)$. Similarly, $A_{r,s}(x):= B_s(x)\setminus B_r(x)$ and  $A^g_{r,s}(x):= B^g_s(x)\setminus B^g_r(x)$ will denote annuli in $\R^2$ and in $(M,g)$.

\

\paragraph{Plan of the paper}
The paper is organized as  follows. In Section \ref{sec:prel} we collect some 
preliminary facts including consequences of the Moser-Trudinger inequality and 
the Green's representation formula. In Section \ref{sec:extract} we extract 
bubbles from a blow-up procedure and derive uniform $L^p$-estimates for $e^{u_n}$ and gradient estimates for $v_n$ away from the bubbles. Section \ref{sec:trees} is devoted to the 
analysis of bubble-clusters and some useful analysis of scales is introduced. 
In Section 
\ref{sec:necks} we derive crucial integral decay estimates on dyadic annuli, 
leading to a local quantization result. In Section \ref{sec:cluster} we then 
 extend the analysis globally, by properly analyzing the clustering behavior of bubbles, and complete the proof of Theorem \ref{trmmain}. 
Finally, in Section \ref{secexample} we give the clustering example, hence proving Theorem \ref{trmclusterexample}, while in 
Section \ref{sec:isolation} we prove isolation of the blow-up points under a stronger integral condition, hence proving Proposition \ref{propisolatedstrong}. 

\begin{center}
	{\bf Acknowledgments} 
\end{center}

\noindent 
The authors  are members of GNAMPA, as part of INdAM. A.M. is also supported by the project {\em Ricerca di base} from Scuola Normale Superiore. 
\

\begin{center}
	{\bf Statement on the use of AI} 
\end{center}

\noindent
The main novelties in this paper that lead to quantization for Palais-Smale sequences (in particular the $L^p$-bubble extraction and the gradient estimates in Section \ref{sec:extract}, and the test function argument and localization on dyadic annuli in Section \ref{sec:necks}) were entirely generated by the authors. 
AI (Claude Opus 5 and ChatGPT 5.4-5.6) was used in the following tasks: 

\smallskip

\noindent
$\bullet$ to construct the bubbling configuration of Theorem \ref{trmclusterexample} and draft the relative estimates;\\
 $\bullet$ to draft some routine proofs and the bubble tree structure, mainly in Sections \ref{sec:trees}, \ref{sec:cluster};\\
 $\bullet$ to proof-check the arguments.

 \smallskip
 
\noindent
All the final text has been revised and proof-checked by the authors.

\section{Preliminary estimates}\label{sec:prel}

The proof of the next lemma, which follows by testing \eqref{PSeq} against $u_n-v_n$, can be found in \cite[Lemma 2.2]{malizia-2025-ACV}.
\begin{lemma}\label{lemclose}
Let $u_n$ and $v_n$ be defined by \eqref{PSeq} and \eqref{vn}. Then
\begin{equation}\label{close}
        \|u_n-v_n\|_{H^1(M)}\to0.
\end{equation}
\end{lemma}

We next derive a simple but useful consequence of the sharp Moser-Trudinger inequality.

\begin{lemma}
\label{lem:MT}
Let $(M,g)$ be a closed Riemannian surface. For every $\beta>0$,
there exists a constant $C=C(M,g,\beta)>0$ such that
\begin{equation}\label{MT1} 
    \int_M
e^{\beta |u-\bar u|}\,dV_g
\leq
C\exp\left(
\frac{\beta^2}{16\pi}
\int_M |\nabla_g u|^2\,dV_g
\right)
\end{equation}

for every $u\in H^1(M)$.

Let $D\subset\mathbb R^2$ be a (connected) bounded smooth domain. For every
$\beta>0$, there exists a constant $C=C(D,\beta)>0$ such that
\begin{equation}\label{MT2} 
\int_D
e^{\beta |u-\bar u_D|}\,dx
\leq
C\exp\left(
\frac{\beta^2}{8\pi}
\int_D |\nabla u|^2\,dx
\right)
\end{equation}
for every $u\in H^1(D)$.
\end{lemma}

\begin{proof}
The first estimate follows from Fontana's inequality on closed surfaces \cite{fontana-mosertrudinger-manifolds},
\[
\sup_{\substack{
u\in H^1(M),\ \bar u=0\\
\int_M|\nabla_g u|^2\,dV_g\leq1
}}
\int_M e^{4\pi u^2}\,dV_g<+\infty.
\]
The second follows from the Chang--Yang inequality \cite{chang-yang-1988-JDG}
\[
\sup_{\substack{
u\in H^1(D),\ \bar u_D=0\\
\int_D|\nabla u|^2\,dx\leq1
}}
\int_D e^{2\pi u^2}\,dx<+\infty.
\]
In both cases, the claimed linear-exponential estimate follows
from Young's inequality
\[
\beta |t|
\leq
\frac{\alpha t^2}{\gamma}
+\frac{\beta^2\gamma}{4\alpha},
\]
with $\gamma=\int|\nabla u|^2$ and $\alpha=4\pi$, respectively
$\alpha=2\pi$.
\end{proof}

\begin{lemma}\label{exppqM}
Suppose $f_n-g_n\to0$ in $H^1(M)$.
If $q>p\ge1$ and $E_n\subseteq M$ are measurable  with
\[
        \int_{E_n} e^{qg_n}\dV\le C,
\]
then
\begin{equation}\label{exppqMclaim}
        \int_{E_n}|e^{pf_n}-e^{pg_n}|\dV\to0.
\end{equation}
\end{lemma}

\begin{proof}
Set $w_n=f_n-g_n$. By the mean-value theorem,
\[
        |e^{pf_n}-e^{pg_n}|\le p|w_n|e^{pg_n}e^{p|w_n|}.
\]
Choose $a,b>1$ with $p/q+1/a+1/b=1$. Holder's inequality gives
\[
\int_{E_n}|e^{pf_n}-e^{pg_n}|\dV
\le p\|w_n\|_{L^a(M)}
        \left(\int_{E_n}e^{qg_n}\dV\right)^{p/q}
        \left(\int_Me^{bp|w_n|}\dV\right)^{1/b}.
\]
The first factor on the right-hand side tends to zero by the Sobolev embedding. The second factor is bounded by assumption. As for the third factor, since $w_n\to0$ in $H^1(M)$, we have $\overline{|w_n|}\to 0$ by Cauchy-Schwarz, so  \eqref{MT1} in Lemma \ref{lem:MT} implies
\[
        \int_Me^{bp|w_n|}\dV= e^{bp\overline{|w_n|}} \int_M e^{bp(|w_n|-\overline{|w_n|})}\dV \le C (1+o(1))\le C,
\]
uniformly in $n$. Hence \eqref{exppqMclaim} follows.
\end{proof}

We aim next to localize the above result to some domains of the plane. 

\begin{lemma}\label{exppqmean}
Let $D\subset\R^2$ be a fixed bounded smooth domain. Assume that $f_n-g_n\to0$ in $H^1(D)$.
If $q>p\ge1$, $E_n\subset D$, and
\[
        \int_{E_n}e^{qg_n}\,dx\le C,
\]
for a fixed constant $C$, then
\[
        \int_{E_n}|e^{pf_n}-e^{pg_n}|\,dx\to0.
\]
\end{lemma}

\begin{proof}
The proof is very similar to the previous one, using \eqref{MT2} instead of \eqref{MT1}.
\end{proof}

\begin{rmk}\label{rmkmetrics}
The same conclusion as in Lemma \ref{exppqmean} holds if the Euclidean measure and gradient are replaced by a sequence of smooth metrics converging in $C^1$ to the Euclidean metric on $D$ and the gradients are taken with respect to such metrics. Indeed, 
if $g_n\to g_{\mathbb R^2}$ uniformly on $\overline D$, then
the corresponding volume elements and Dirichlet energies are uniformly
equivalent to their Euclidean counterparts. Hence the constants in the
preceding estimate may be chosen independently of $n$.
\end{rmk}

The next properties of the Green function are well-known, see e.g.  Chapter 4 in \cite{aubinbook}. 

\begin{lemma}\label{green}
Let $G$ be the Green function of $-\Delta_g$ on $M$, normalized by
\[
        -\Delta_{g,y}G(x,y)=\delta_x-1,
        \qquad \int_MG(x,y)\dV(y)=0.
\]
Then
\begin{equation}\label{greenest}
        \left|G(x,y)-\frac1{2\pi}\log\frac1{d_g(x,y)}\right|\le C,
        \qquad
        |\nabla_xG(x,y)|\le \frac{C}{d_g(x,y)}
\end{equation}
for $x\ne y$. If $-\Delta_gu=f$ with $\int_Mf\dV=0$, then
\begin{equation}\label{greenrep}
        u(x)=\bar u+\int_MG(x,y)f(y)\dV(y),
        \qquad \bar u:=\int_Mu\dV.
\end{equation}
\end{lemma}

The next result is based on a classical estimate from \cite{brezis-merle-1991}, which is reflected also in \cite{ohtsukasuzuki-2003-Calcvar}. For its proof, see e.g. \cite[Lemma 2.4]{malizia-2025-ACV}. 

\begin{lemma}\label{cc}
Let $\zeta_n$ solve $-\Delta_g\zeta_n=f_n$ on $M$, $\|f_n\|_{L^1(M)}\le C$.
Then, up to a subsequence, either
\begin{enumerate}[label=\textup{(\alph*)}]
\item there exist $q>1$ and $C>0$ such that
\[
        \int_Me^{q(\zeta_n-\bar\zeta_n)}\dV\le C, 
        \qquad \hbox{ or }
\]
\item  there are finitely many points $a_1,\dots,a_N\in M$ such that for every $r>0$ and every $j$
\[
        \liminf_{n\to\infty}\int_{B^g_r(a_j)}|f_n|\dV\ge4\pi.
\]
\end{enumerate}
\end{lemma}

Let now again $(u_n)$ be as in \eqref{PSeq}-\eqref{normalization}: we are going to show next an alternative to the compactness behaviour, again in the spirit of \cite[Thm. 3]{brezis-merle-1991} or \cite[Thm. 1.1]{ohtsukasuzuki-2003-Calcvar}.

\begin{lemma}\label{initialconcentration}
If $(u_n)$ blows up, then the measures $\lambda_nhe^{u_n}\dV$ concentrate somewhere. More precisely, there exists $a\in M$ such that for every $r>0$,
\[
        \liminf_{n\to\infty}\int_{B^g_r(a)}\lambda_nhe^{u_n}\dV\ge4\pi.
\]
In particular, $\lambda_\infty\ge 4\pi$. 
\end{lemma}

\begin{proof}
Apply Lemma~\ref{cc} to $\zeta_n=v_n$, with $v_n$ as in \eqref{vn}, and $f_n=\lambda_n(he^{u_n}-1)$. Suppose by contradiction that  compactness  holds. Then there are $q>1$ and $C>0$ such that
\[
        \int_Me^{q(v_n-\bar v_n)}\dV\le C.
\]
The normalization \eqref{normalization} and Jensen's inequality give $\bar v_n = \bar u_n\le C$. After passing to a subsequence, either $\bar v_n\to-\infty$, or $\bar v_n$ is bounded from below.

In the first case, $e^{v_n}\to0$ in $L^p(M)$ for every $1\le p<q$. Lemma~\ref{exppqM}, applied with exponent $1$, gives $e^{u_n}-e^{v_n}\to0$ in $L^1(M)$, contradicting $\int_Mhe^{u_n}\dV=1$.

In the second case, $e^{v_n}$ is bounded in $L^q(M)$. Fix $p\in(1,q)$. Lemma~\ref{exppqM}, now applied with exponent $p$, gives a uniform $L^p$ bound for $e^{u_n}$. Equation~\eqref{vn} and elliptic regularity imply that $(v_n)$ is bounded in $W^{2,p}(M)$ and hence precompact in $H^1(M)$. Lemma~\ref{lemclose} then implies that $(u_n)$ is precompact in $H^1(M)$, contradicting blow-up.

Therefore the concentration alternative in Lemma~\ref{cc} holds. Since the negative part of $\lambda_n(he^{u_n}-1)$ is uniformly absolutely continuous with respect to volume, the concentration is carried by $\lambda_nhe^{u_n}\dV$.

The last claim follows at once from the normalization \eqref{normalization}.
\end{proof}

\begin{rmk}
In the proof of Lemma \ref{initialconcentration} we considered both the cases $\bar v_n\to -\infty$ and $\bar v_n \ge -C$. For a blowing-up Palais-Smale sequence $(u_n)$ with $J_{\lambda_n}(u_n)\to c$, the second case is always ruled out since $\bar u_n\to -\infty$. Indeed, for a blowing-up sequence $(u_n)\subset H^1(M)$, $\|\nabla_g u_n\|_{L^2}\to+\infty$, otherwise the Moser-Trudinger inequality gives uniform bounds on $e^{pu_n}$ in $L^1(M)$ for every $1\le p<+\infty$, which easily implies that a subsequence $(u_{n'})$ converges in $H^1(M)$. As we are not assuming  $J_{\lambda_n}(u_n)$  bounded, we have to consider both cases.
\end{rmk}

\section{Extraction of bubbles} \label{sec:extract}

In the following proposition we perform an exhaustive extraction of bubbles, in the sense that, inductively, we extract a new bubble as in \eqref{bubbleextractionprofile}-\eqref{defVi} below, until the uniform estimate \eqref{nofurther} is satisfied. This estimate can be seen as an $L^p$ version interpolating between the case $p=1$ treated in Section 5 of \cite{malizia-2025-ACV}, and the case $p=+\infty$ which holds for exact solutions of the mean-field equation (see \cite[Eq. (9)]{lishafrir-1994}, or \cite[Theorem 0.2(c)]{li-1999-CMP-harnacktype}), giving uniform bounds on the quantity
\begin{equation}\label{nobounds}
d_n^2(x)e^{u_n(x)},
\end{equation}
where $d_n$ is defined in \eqref{defdn} below.

On the one hand, such $L^\infty$ bounds do not hold for Palais-Smale sequences, as the error term $R_n\to 0$ in $H^{-1}(M)$ can uncontrollably alter $u_n$ pointwise; on the other hand, the $L^1$ bounds of \cite{malizia-2025-ACV} are not sufficient for the gradient estimates  we will need later. Hence the need for the $L^p$-version, with $p>2$. Notice that if the constant $C_*$ below could be chosen independent of $p$, as $p\to +\infty$ the bound \eqref{nofurther}
would imply an $L^\infty$ bound on \eqref{nobounds}.

\begin{prop}\label{propextraction}
Assume that $(u_n)$ blows up. Fix $p\in(2,+\infty)$ and $\rho_*>0$. Then, up to a subsequence, there exist an integer $\ell\geq1$, points $x_{i,n}\to x^{(i)}\in M$, radii $r_{i,n}\downarrow0$, constants $c_{i,n}\in\mathbb R$, and a constant $C_*>0$ such that the following properties hold.

\begin{enumerate}[label=\textup{(\roman*)}]
\item For every $R>0$,
\begin{equation}\label{separ}
        B^g_{Rr_{i,n}}(x_{i,n})\cap B^g_{Rr_{j,n}}(x_{j,n})=\emptyset,
        \qquad i\ne j,
\end{equation}
for all sufficiently large $n$, and for some $0< \alpha <1-\frac{2}{p}$,
\begin{equation}\label{bubbleextractionprofile}
     \widetilde V_{i,n}:=   v_n(\exp_{x_{i,n}}(r_{i,n}\cdot))+2\log r_{i,n}+c_{i,n}
        \to V_i
        \quad\hbox{in }C^{1,\alpha}_{\loc}(\R^2),
\end{equation}
where for some $\delta_i>0$, $\mu_i>0$, $x_{0,i}\in \R^2$, $V_i$ has the form
\begin{equation}\label{defVi}
V_i(x)=\log \left(\frac{8\delta_i^2}{\mu_i(\delta_i^2+|x-x_{0,i}|^2)^2}\right),
\end{equation}
and satisfies
\begin{equation}\label{volVi}
-\Delta V_i=\mu_i e^{V_i}\quad \text{in }\R^2,\quad \text{and}\quad \int_{\R^2}\mu_i e^{V_i}=8\pi,
\end{equation}
with $\mu_i=\lambda_\infty h(x^{(i)})$. Moreover,
\begin{equation}\label{convun8pi}
\lim_{R\to\infty}\lim_{n\to\infty}\int_{B^g_{Rr_{i,n}}(x_{i,n})}\lambda_n h e^{u_n}dV_g= \lim_{R\to\infty}\lim_{n\to\infty}\int_{B^g_{Rr_{i,n}}(x_{i,n})}\lambda_n h e^{v_n+c_{i,n}}dV_g=8\pi.
\end{equation}
\item Setting
\begin{equation}\label{defdn}
 d_n(x):=\min_{1\le i\le\ell}d_g(x,x_{i,n}),
\end{equation}
we have, for every $n$ sufficiently large,
\begin{equation}\label{nofurther}
 r^{2}
 \left(
 \fint_{B^g_r(x)}e^{pu_n}\,dV_g
 \right)^{1/p}
 \le \rho_*
 \qquad
 \text{whenever }
 0<r\le \frac{d_n(x)}{C_*}. 
\end{equation}
\end{enumerate}
\end{prop}

\begin{proof}
We divide the proof into three steps.

\emph{Step 1.}
By Lemma~\ref{initialconcentration}, there exist $a\in M$ and radii $\varrho_n\downarrow0$ such that
\[
\int_{B^g_{\varrho_n}(a)} \lambda_nhe^{u_n}\,\dV \geq 2\pi
\]
for all sufficiently large $n$. Hence
\[
\int_{B^g_{\varrho_n}(a)}e^{u_n}\,\dV \geq \frac{2\pi} {(\sup_n\lambda_n)\|h\|_{L^\infty(M)}}.
\]
By Hölder's inequality and the estimate $\operatorname{Vol}_g(B^g_r(x))\leq C_g r^2$, there exists $\kappa_*>0$, independent of $n$, such that
\[ \varrho_n^2 \left( \fint_{B^g_{\varrho_n}(a)} e^{pu_n}\,\dV \right)^{1/p} \geq\kappa_*.
\]
Define
\[
m_n(r) := \sup_{x\in M} r^2 \left( \fint_{B^g_r(x)}e^{pu_n}\,\dV \right)^{1/p}.
\]
For every fixed $R\geq1$ and all sufficiently large $n$, since $\Vol_g(B_r^g)/\Vol_{g_E}(B_r)\to 1$ as $r\to 0^+$, we got
\[
\begin{aligned}
m_n(R\varrho_n) &\geq (R\varrho_n)^2 \left( \fint_{B^g_{R\varrho_n}(a)} e^{pu_n}\,\dV \right)^{1/p}\\ &\geq R^{2-\frac2p} \varrho_n^2 \left( \fint_{B^g_{\varrho_n}(a)} e^{pu_n}\,\dV \right)^{1/p}(1+o(1))\\ &\geq \kappa_*R^{2-\frac2p}(1+o(1)). \end{aligned} \] Since $\rho_*>0$ is fixed, we may choose $R_*\geq1$ such that $\kappa_*R_*^{2-\frac2p}>\rho_*$.
Set $s_n:=R_*\varrho_n$.
Then $s_n\downarrow0$ and $m_n(s_n)>\rho_*$ for all sufficiently large $n$. Since $m_n(r)\to0$ as $r\downarrow0$, define
\[
r_n := \inf\bigl\{ r\in(0,s_n]: m_n(r)\geq\rho_* \bigr\}.
\]
By continuity, $ m_n(r_n)=\rho_*, $ and, since $r_n\leq s_n\to0$, we get
$r_n\to0.$
Choose now $x_n\in M$ such that
\begin{equation}\label{eq:firstscale}
r_n^{2} \left( \fint_{B^g_{r_n}(x_n)} e^{pu_n}\,\dV \right)^{1/p} =\rho_*.
\end{equation}
By compactness, up to a subsequence, we can further assume that
\begin{equation}\label{convxn}
x_n\to x^{(1)}\in M.
\end{equation}

Let
\[
        F_n(y):=\exp_{x_n}(r_ny),
        \qquad \hat g_n:=r_n^{-2}F_n^*g.
\]
By standard properties of the exponential map,
\begin{equation}\label{convgn}
\hat g_n\to g_{\R^2}\quad \text{in }C^2_{\loc}(\R^2).
\end{equation}
Set
\[
        U_n(y):=u_n(F_n(y))+2\log r_n,
        \qquad
        V_n(y):=v_n(F_n(y))+2\log r_n.
\]
Then, on $B_{\inj(M)/(2r_n)}(0)$,
\begin{equation}\label{rescaledV}
        -\Delta_{\hat g_n}V_n
        =\lambda_n h(F_n(y))e^{U_n(y)}-\lambda_nr_n^2.
\end{equation}
For every $R>0$ and every ball $B_1(z)\subset B_R(0)\subset \R^2$, we can cover $B_1(z)$ with a fixed number (independent of $z$) of balls of radius $1/2$. Using the minimality of $r_n$, the definition of $m_n$ and \eqref{convgn}, we then get
\begin{equation}\label{localLpU}
        \left(\fint_{B_1(z)}e^{pU_n}\,dy\right)^{1/p}\le C \rho_*.
\end{equation}

Fix now  $R\ge 2$ and choose constants $b_n$ such that
\[
        \fint_{B_{2R}}(U_n- \widetilde V_n )\,dy=0,
        \qquad \widetilde V_n:=V_n+b_n.
\]
Using \eqref{close} and the change of variables $x=F_n(y)$,
\[
        \int_{B_{2R}}|\nabla(U_n-\widetilde V_n)|^2\,dy
        \le C\int_{F_n(B_{2R})}|\nabla_g(u_n-v_n)|^2\dV=o(1).
\]
Poincar\'e's inequality gives
\begin{equation}\label{H1closeScaled}
        \|U_n-\widetilde V_n\|_{H^1(B_{2R})}\to0.
\end{equation}
In the following we will borrow ideas from \cite{Adimurthi2006}, see also \cite{martinazzi-2009-JFA}. Let $B_1(z)\subset B_R(0)$ and let $W_n$ solve
\[
        -\Delta_{\hat g_n}W_n=\lambda_nh(F_n)e^{U_n}-\lambda_nr_n^2
        \quad\hbox{in }B_1(z),
        \qquad W_n=0\quad\hbox{on }\partial B_1(z).
\]
By \eqref{localLpU}, the fact that $p>2$, and the uniform ellipticity of $\Delta_{\hat g_n}$, we have 
\begin{equation}\label{WnC1}
        \|W_n\|_{C^1(B_1(z))}\le C(R,\rho_*,p).
\end{equation}
Set $Z_n:=\widetilde V_n-W_n$. Then $Z_n$ is $\hat g_n$-harmonic in $B_1(z)$. Moreover \eqref{localLpU}, \eqref{H1closeScaled}, and \eqref{WnC1} imply
\begin{equation}\label{Zplus}
        \|Z_n^+\|_{L^1(B_1(z))}\le C(R,\rho_*,p).
\end{equation}
Indeed $\widetilde V_n=U_n+o_{H^1}(1)$ and $e^{pU_n}$ is locally bounded in $L^1$, so the positive part cannot have unbounded integral on a fixed ball.

Now set
$$\beta_n:=\|Z_n\|_{L^1(B_{1/2}(z))}.$$
Up to extracting a subsequence, there are two alternatives.

\medskip

\noindent\textbf{Case 1: $\beta_n\le C$.} Then elliptic  estimates (with respect to $\Delta_{\hat g_n}$) give uniform $C^{1,\alpha}$-bounds for $Z_n$, hence also for $\widetilde V_n$, in $B_{1/4}(z)$.

\medskip

\noindent\textbf{Case 2: $\beta_n\to+\infty$.}
Since $\Delta_{\hat g_n}Z_{n}=0$, the positive part
$Z_{n}^{+}$ is weakly $\hat g_n$-subharmonic. Indeed, applying
$\Delta_{\hat g_n}$ to a smooth convex nondecreasing approximation
of $t\mapsto t^{+}$ and passing to the limit gives
\[
        \Delta_{\hat g_n}Z_{n}^{+}\ge0,\quad \text{in }B_1(z)
\]
in the sense of distributions. Then, by elliptic estimates, see e.g. \cite[Thm. 8.17, 8.20]{gilbargtrudinger}: 
\[
        \sup_{B_{3/4}(z)}Z_{n}
        \le
        \sup_{B_{3/4}(z)}Z_{n}^{+}
        \le
        C\|Z_{n}^{+}\|_{L^1(B_1(z))}
        \le A.
\]
Set $H_{n}:=A+1-Z_{n}$. Then $H_{n}$ is positive and
$\hat g_n$-harmonic in $B_{3/4}(z)$, and the uniform Harnack
inequality gives
\[
        \sup_{B_{1/2}(z)}H_{n}
        \le C_H\inf_{B_{1/2}(z)}H_{n}.
\]
This implies that  $\beta_n\asymp H_{n}(z)$.
Hence, after passing to a subsequence,
\[
        \frac{H_{n}}{\beta_n}\to\Psi
        \quad\text{in }C^1(B_{1/4}(z)),
\]
where $\Psi$ is a positive nontrivial harmonic function. Therefore
\[
        \frac{Z_{n}}{\beta_n}\to-\Psi
        \quad\text{in }C^1(B_{1/4}(z)),
\]
and in particular
\begin{equation}\label{negativealternative}
        Z_n\to-\infty\quad\hbox{uniformly in }B_{1/4}(z).
\end{equation}
Then also 
\begin{equation}\label{negativealternative2}
        \tilde V_n\to-\infty\quad\hbox{uniformly in }B_{1/4}(z).
\end{equation}

\medskip

Case $1$ and Case 2 are incompatible on overlapping balls; hence either Case 1 holds for every $B_1(z)\subset B_R(0)$ and $\widetilde V_n$ is locally uniformly bounded in $C^{1,\alpha}(B_{R-1}(0))$, or Case 2 holds for every $B_1(z)\subset B_R(0)$ and $\widetilde V_n$ converges to $-\infty$ locally uniformly in $B_{R-1}(0)$.

The latter alternative is impossible. Indeed, in this case, we can take $q>p$ and have,
\[
        \int_{B_1(0)}e^{q\widetilde V_n}\,dy\to0.
\]
Then, by \eqref{H1closeScaled} and Lemma~\ref{exppqmean},
\[
        \int_{B_1(0)}|e^{pU_n}-e^{p\widetilde V_n}|\,dy\to0.
\]
Thus
$$\int_{B_1(0)}e^{pU_n}\,dy\to0,$$
contradicting \eqref{eq:firstscale}. Therefore $\widetilde V_n$ (the whole sequence) is bounded in $C^{1,\alpha}_{\loc}(\R^2)$.

Observe that $b_n=b_{n,R}$ was defined depending on $R$. Now,
to make the normalization independent of the radius, fix
$D=B_2(0)$ and define
\[
        c_n:=\fint_D(U_n-V_n)\,dy,
\]
and rename $\widetilde V_n:=V_n+c_n$.
Then
\[
        \fint_D\bigl(U_n-\widetilde V_n\bigr)\,dy=0.
\]
Since $\nabla(U_n-V_n)\to0
        \quad\hbox{in }L^2_{\loc}(\R^2)$,
the Poincar\'e inequality gives
\begin{equation}\label{UnVn0}
U_n-\widetilde V_n \to 0
        \quad\hbox{in }H^1_{\loc}(\R^2).
\end{equation}Moreover, if
\[
        b_{n,R}:=\fint_{B_{2R}(0)}(U_n-V_n)\,dy
\]
is the previous normalization, then       $b_{n,R}-c_n\to0$
for every fixed $R$. Hence the local bounds obtained for
$V_n+b_{n,R}$ also hold for $\widetilde V_n=V_n+c_n$. A diagonal extraction therefore gives
\begin{equation}\label{diagbubble2}
        \widetilde V_n\to V\quad\hbox{in }C^{1,\alpha}_{\loc}(\R^2).
\end{equation}

Moreover, \eqref{rescaledV} implies
\[
 -\Delta_{\hat g_n}\widetilde V_n
 =
 \lambda_nh(F_n)e^{U_n}-\lambda_nr_n^2.
\]
For every fixed $R>0$,   $F_n(y)=\exp_{x_n}(r_ny)\to x^{(1)}$
uniformly for $y\in B_R(0)$, hence
\[
        \lambda_nh\circ F_n\to
        \lambda_\infty h(x^{(1)})=:\mu
\]
locally uniformly in $\R^2$. In particular, since $\lambda_\infty>0$ by Lemma \ref{initialconcentration}, we have $\mu>0$.

Now \eqref{UnVn0}, \eqref{diagbubble2} and Lemma \ref{exppqmean} give
\begin{equation}\label{eUnetildeVn}
     e^{U_n}-e^{\widetilde V_n}
        \to 0
        \quad\hbox{in }L^1_{\loc}(\R^2),
\end{equation}
hence $e^{U_n}\to e^V$ in $L^1_{\loc}(\R^2)$. Using
$\hat g_n\to g_{\R^2}$ in $C^2_{\loc}$ and
$\lambda_nr_n^2\to0$, we obtain
\[
        -\Delta V=\mu e^V
        \quad\hbox{in }\R^2,
        \qquad
        \mu=\lambda_\infty h(x^{(1)})>0.
\]
Finally, since
\[
        dV_{\hat g_n}=r_n^{-2}F_n^*(dV_g),
        \qquad
        e^{U_n}=r_n^2e^{u_n\circ F_n},
\]
we have, for every fixed $R>0$,
\begin{align*}
 \int_{B_R(0)}
 h(F_n)e^{U_n}\,dV_{\hat g_n}
 =
 \int_{F_n(B_R(0))}he^{u_n}\,dV_g 
 \le
 \int_Mhe^{u_n}\,dV_g=1.
\end{align*}
Passing to the limit gives
\[
        h(x^{(1)})\int_{B_R(0)}e^V\,dy\le1.
\]
Letting $R\to+\infty$, we conclude that
\[
        \int_{\R^2}e^V\,dy<+\infty,
        \qquad
        \int_{\R^2}\mu e^V\,dy\le\lambda_\infty.
\]
By Proposition \ref{classification}, $V$ is as in \eqref{defVi} and \eqref{volVi} holds.

Now, using \eqref{eUnetildeVn} and \eqref{diagbubble2} we infer
\[\begin{split}\lim_{n\to+\infty}\int_{B_{Rr_n}^g(x_n)}\lambda_n he^{u_n}dV_g &=\lim_{n\to+\infty}\int_{B_R(0)}\lambda_n h(F_n) e^{U_n}dx\\
&= \lim_{n\to+\infty}\int_{B_R(0)}\lambda_n h(F_n)e^{\widetilde V_n}dx=\int_{B_R}\mu e^Vdx.
\end{split}\]
Letting $R\to \infty$ and observing that 
\[\int_{B_R(0)}\lambda_n h(F_n)e^{\widetilde V_n}dx=\int_{B^g_{Rr_n}(x_n)}\lambda_n he^{v_n+c_n}dx\]
completes the proof of \eqref{convun8pi}.

\medskip

\emph{Step 2.} We now proceed by induction. Let $r_{1,n}:= r_n$ and $x_{1,n}:= x_n$ be as built in Step 1. For any $k\ge 1$ suppose that bubbles $(x_{i,n},r_{i,n})$, $i=1,\dots,k$, have already been extracted and satisfy \eqref{separ}-\eqref{convun8pi}. Define
\[
        d_{k,n}(x):=\min_{1\le i\le k}d_g(x,x_{i,n}).
\]
If \eqref{nofurther} fails for all $C_*>0$, then one can find $L_n\to\infty$, points $y_n\in M$, and scales $s_n>0$ such that
\begin{equation}\label{newscale}
        s_n\le \frac{d_{k,n}(y_n)}{L_n}\to 0,
        \qquad
        s_n^{2}\left(\fint_{B^g_{s_n}(y_n)}e^{pu_n}\dV\right)^{1/p}=\rho_*,
\end{equation}
and $s_n$ is minimal among all choices of $(y_n,s_n)$ satisfying \eqref{newscale}.

Rescale around $y_n$ by $s_n$ and, similar to Step 1, write
\[
 F_n(w):=\exp_{y_n}(s_nw),\qquad
 \widehat g_n:=s_n^{-2}F_n^*g,\qquad
 U_n(w):=u_n(F_n(w))+2\log s_n.
\]
Fix $R\ge2$. We claim that, for all sufficiently large $n$ and every
 ball $B_1(z)\subset B_R(0) \subset \R^2$,
\begin{equation}\label{localLpUind}
 \left(
 \fint_{B_1(z)}e^{pU_n}\,dw
 \right)^{1/p}
 \le C\rho_*.
\end{equation}
By minimality of $s_n$ we have
\begin{equation}\label{minimalityconsequence}
0<r\le s_n,
 \quad
 r\le\frac{d_{k,n}(x)}{L_n}\quad \Rightarrow\quad  r^2
 \left(
 \fint_{B_r^g(x)}e^{pu_n}\,dV_g
 \right)^{1/p}
 \le\rho_*.
\end{equation}
Let now $B_1(z)\subset B_R(0)$. For every $w\in B_1(z)$, the
$1$-Lipschitz property of $d_{k,n}$ and the definition of $s_n$ give,
provided $L_n\ge2R$,
\[
\begin{split}
 d_{k,n}(F_n(w))
 &\ge d_{k,n}(y_n)-d_g(F_n(w),y_n)\\
 &=d_{k,n}(y_n)-s_n|w|\\
 &\ge L_ns_n-Rs_n
 \ge\frac12L_ns_n.
\end{split}
\]
Consequently
\[
 \frac{s_n}{2}
 \le\frac{d_{k,n}(F_n(w))}{L_n}.
\]

Cover $B_1(z)$ by a fixed number $N$ of Euclidean balls
$B_{1/4}(z_j)$, with $z_j\in B_1(z)$. Since
$\widehat g_n\to g_{\mathbb R^2}$ uniformly on $B_R(0)$, for all
sufficiently large $n$,
\[
 F_n\bigl(B_{1/4}(z_j)\cap B_1(z)\bigr)
 \subset B^g_{s_n/2}(F_n(z_j)).
\]
Applying \eqref{minimalityconsequence} with
$x=F_n(z_j)$ and $r=s_n/2$, we obtain
\[
 \int_{B^g_{s_n/2}(F_n(z_j))}
 e^{pu_n}\,dV_g
 \le C_g
 \rho_*^p\left(\frac{s_n}{2}\right)^{2-2p}.
\]
Therefore
\[
 \int_{B_1(z)}e^{pU_n}\,dV_{\widehat g_n}
 =
 s_n^{2p-2}
 \int_{F_n(B_1(z))}e^{pu_n}\,dV_g\\
 \le N C_g
 2^{2p-2}\rho_*^p.
\]
Finally, the volume densities of $\widehat g_n$ converge uniformly to
$1$ on $B_R(0)$, and hence \eqref{localLpUind} follows.
The constants are independent of the center $z$.

Thus the argument of Step~1 applies and $(x_{k+1,n},r_{k+1,n}):=(y_n,s_n)$ produces a new
bubble satisfying \eqref{bubbleextractionprofile}-\eqref{convun8pi}. Up to a subsequence, we also have $x_{k+1,n}\to x^{(k+1)}\in M$.

We need to verify that \eqref{separ} holds as well for $1\le i\le k$, $j=k+1$.
First notice that, similar to Step 1, the above procedure implies that for every fixed $R>0$
\begin{equation}\label{boundpUnR}
\int_{B_R(0)}e^{pU_{i,n}} dy \le C(R).
\end{equation}
Now assume by contradiction that $d_g(y_n,x_{i,n})\leq C r_{i,n}$ for some constant $C>0$. Then
\begin{align*}
    s_n\leq \frac{d_{k,n}(y_n)}{L_n}\leq \frac{d_g(y_n,x_{i,n})}{L_n}\leq \frac{C r_{i,n}}{L_n} \quad \Longrightarrow \quad \frac{s_n}{r_{i,n}}\to 0.
\end{align*}
Hence $B_{s_n}^g(y_n)\subset B^g_{(C+1)r_{i,n}}(x_{i,n})$, which together with \eqref{newscale}, \eqref{boundpUnR} and the inductive assumption implies
\begin{align*}
    \rho_*=s_n^2\Big(\fint_{B^g_{s_n}(y_n)}e^{p u_n}\,dV_g\Big)^{\frac{1}{p}}&\leq C_g s_n^{2-\frac{2}{p}}\Big(\int_{B^g_{(C+1)r_{i,n}}(x_{i,n})}e^{p u_n}\,dV_g\Big)^{\frac{1}{p}}\\
    &\leq C_g \Big(\frac{s_n}{r_{i,n}}\Big)^{2-\frac{2}{p}}\Big(\int_{B_{C+1}(0)} e^{pU_{i,n}}\,dy\Big)^{\frac{1}{p}}\to0,
\end{align*}
which is a contradiction (here $U_{i,n}(y):=u_n(\mathrm{exp}_{x_{i,n}}(r_{i,n}y))+2\log r_{i,n}$ and the last integral is bounded by \eqref{localLpU} applied in the inductive step). This proves \eqref{separ}.

\medskip

\emph{Step 3.}
It remains to prove that the induction terminates after finitely many
steps. Suppose by contradiction that the extraction procedure does not
terminate. Then, for every $N\in\N$, up to a subsequence, we obtain $N$ bubbles
$(x_{i,n},r_{i,n})$ satisfying \eqref{separ} and \eqref{convun8pi}, so that the normalization \eqref{normalization} implies
\[
\begin{split}
\lambda_\infty=\lim_n\int_M \lambda_n he^{u_n} dV_g\ge \lim_{R\to+\infty} \lim_{n\to+\infty}\int_{\bigcup_{i=1}^n B^g_{Rr_{i,n}}(x_i)} \lambda_n h e^{u_n}dV_g= 8\pi N,
\end{split}\]
contradicting the boundedness of $\lambda_n$.

Since the failure of \eqref{nofurther} implies the existence of a new bubble, when the induction procedure stops, 
\eqref{nofurther} holds.
\end{proof}

The integral bounds in part (ii) of Proposition \ref{propextraction} allow us to derive a gradient bound for the sequence $(v_n)$. We will follow a method of
\cite{druet-robert-2006-PAMS}, but replacing the pointwise estimates on $u_n$ (which do not hold in our setting) with the $L^p$-estimate \eqref{nofurther}.

\medskip
 
From now on $\{(x_{i,n},r_{i,n})\}_{i=1}^\ell$ denotes the complete family obtained in Proposition~\ref{propextraction}, and $d_n(x)=\min_i d_g(x,x_{i,n})$.

\begin{prop}\label{propgradient}
There exists $C>0$ such that
\begin{equation}\label{gradestM}
        d_n(x)|\nabla_gv_n(x)|\le C
        \qquad\hbox{for all }x\in M.
\end{equation}
\end{prop}

\begin{proof}
From the Green's representation formula, we have 
\[
        \nabla_gv_n(x)=\int_M\nabla_xG(x,y)\lambda_n(he^{u_n(y)}-1)\dV(y).
\]
By Lemma \ref{green}, the contribution of the constant term is uniformly bounded. Using \eqref{greenest}, we then find 
\[
        |\nabla_gv_n(x)|\le C+C\int_M\frac{\lambda_nhe^{u_n(y)}}{d_g(x,y)}\dV(y).
\]
Let $L=C_*$. Split the integral into $B^g_{d_n(x)/L}(x)$ and its complement. On the complement,
\[
 d_n(x)\int_{M\setminus B^g_{d_n(x)/L}(x)}\frac{\lambda_nhe^{u_n(y)}}{d_g(x,y)}\dV(y)
 \le CL\int_M\lambda_nhe^{u_n}\dV\le C.
\]
On the inner ball, \eqref{nofurther} gives
\[
        \|e^{u_n}\|_{L^p(B^g_{d_n(x)/L}(x))}
        \le C\left(\frac{d_n(x)}{L}\right)^{\frac{2}{p}-2}.
\]
Since $p>2$, H\"older's inequality gives
\[
\begin{split}
 d_n(x)\int_{B^g_{d_n(x)/L}(x)}\frac{\lambda_nhe^{u_n(y)}}{d_g(x,y)}\dV(y)
 &\le C d_n(x)
 \|d_g(x,\cdot)^{-1}\|_{L^{p'}(B^g_{d_n(x)/L}(x))}
 \|e^{u_n}\|_{L^p(B^g_{d_n(x)/L}(x))}  \\
 &\le C d_n(x)
 \left(\frac{d_n(x)}{L}\right)^{\frac{2}{p'}-1}
 \left(\frac{d_n(x)}{L}\right)^{\frac{2}{p}-2}
 \le C.
\end{split}
\]
Combining the two estimates proves \eqref{gradestM}.
\end{proof}

\section{Cluster trees}\label{sec:trees}

Since the centers $x_{i,n}\to x^{(i)}$ ($1\le i\le \ell$) extracted in Proposition~\ref{propextraction} need not have distinct limits,
we start calling $\{a_1,\dots, a_N\}$ the distinct limiting points of the sequence extracted in Proposition \ref{propextraction}, and for each $a_\alpha$ we set
$$I(\alpha):=\{i\in \{1,\dots,\ell\}: x^{(i)}=a_\alpha\}.$$

Then, we organize the bubbling sequences $\{(x_{i,n},r_{i,n}):i\in I(\alpha)\}$ converging to the same blow-up point $a_\alpha$ into a finite cluster tree, based on the the bubble scales $r_{i,n}$ and the distances $d_g(x_{i,n},x_{j,n})$, $i,j\in I(\alpha)$.

We start giving the basic definitions of an abstract rooted tree.

\begin{defin}
A \emph{finite rooted tree} is a finite partially ordered set
$(\mathcal T,\preceq)$ with a unique maximal element
$A_{\mathrm{root}}$, called \emph{root}, such that every node $A\in \mathcal{T}$
 with $A\neq A_{\mathrm{root}}$ has a unique \emph{predecessor}  
$\operatorname{pred}(A)$, i.e. an element such that
\[
 A\preceq \operatorname{pred}(A), \quad A\ne \operatorname{pred}(A),
\]
and there is no $B\in \mathcal{T}\setminus \{A, \operatorname{pred}(A)\}$ such that $A\preceq B\preceq \operatorname{pred}(A)$.

We say that a node $B$ is a \emph{branch} of $A$ if
$\operatorname{pred}(B)=A$.
The set of branches of $A$ is denoted by $\operatorname{br}(A)$.
A node having no branches is called a \emph{leaf}. A node which is not
a leaf, nor the root is called an \emph{internal node}.
\end{defin}

Fix $1\le \alpha\le N$, and $I(\alpha)\subseteq \{1,\dots,\ell\}$ as above. The corresponding tree $\mathcal{T}(\alpha)$ will be made of subsets of $I(\alpha)$. The order will be given by inclusion: $A\preceq B$ if and only if $A\subseteq B$. We set $A_{\mathrm{root}}=I(\alpha)$ and the leaves are the singletons $\{i\}$ for $i\in I(\alpha)$. For each node $A$, $\operatorname{br}(A)$ forms a partition of $A$.

For a node $A$, set
\[
 X_{A,n}:=\{x_{i,n}:i\in A\},
 \qquad
 \diam_g X_{A,n}
 :=\max_{i,j\in A}d_g(x_{i,n},x_{j,n}).
\]
Fix an arbitrary index $i_A\in A$ and define
\[
 x_{A,n}:=x_{i_A,n}.
\]
The \emph{inner scale} of $A$ is defined as
\[
 s_{A,n}:=
 \begin{cases}
 r_{i,n},& \text{if }A=\{i\},\\[1mm]
 \diam_g X_{A,n},& \text{if }|A|\ge2.
 \end{cases}
\]
The \emph{outer scale} is defined as follows. If $A$ is not the root, then
\[
 S_{A,n}
 :=
 \frac12
 \dist_g\bigl(X_{A,n},
              X_{\operatorname{pred}(A)\setminus A,n}\bigr),
\]
where
\[
 \dist_g(X_{A,n},X_{B,n})
 :=
 \min_{\substack{i\in A\\ j\in B}}
 d_g(x_{i,n},x_{j,n}).
\]
Thus $S_{A,n}$ is half the distance from $A$ to the other branches of
its predecessor.

If $A=A_{\mathrm{root}}=I(\alpha)$, we will choose a fixed  $S_\alpha>0$ sufficiently small and such that
\[
0<S_\alpha\le \frac{1}{8}\min_{\beta\ne\alpha} \dist_g(a_\alpha,a_\beta),\quad S_\alpha\le S_0, \quad S_0\le \frac{1}{2}\mathrm{inj}(M,g),\]
where $S_0$ is as in Lemma \ref{lemaverage} with $\nu_1\in (1,2)$ and $\nu\in (1,\nu_1)$ fixed (for instance $\nu_1=3/2$, $\nu=5/4$) throughout the paper, and set
\[
 S_{A,n}:=S_\alpha\quad \text{for every }n\in\mathbb{N}.
\]

\begin{figure}
    \centering
    \includegraphics[width=1\textwidth]{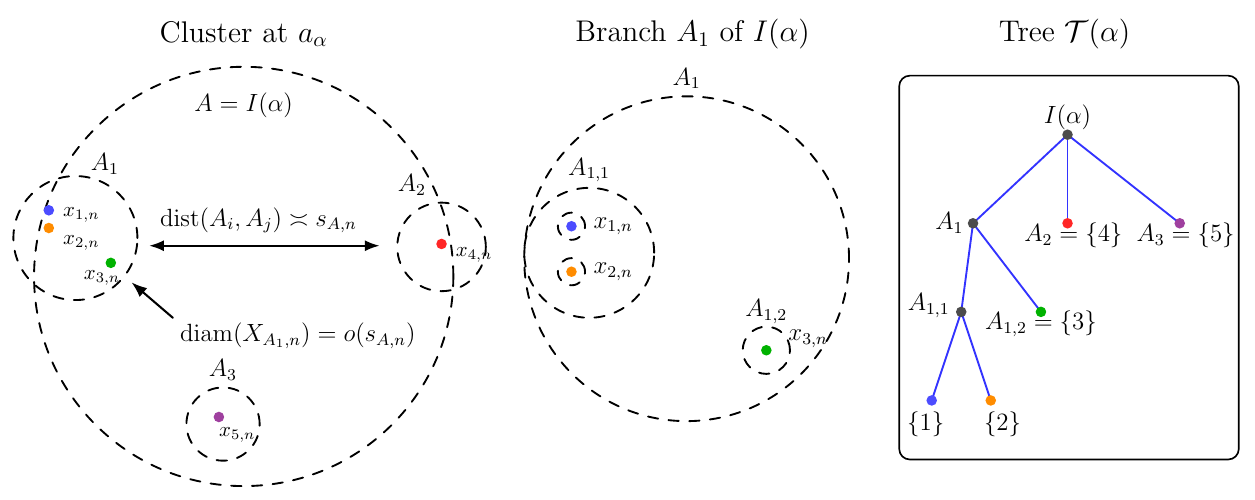}
    \caption{Example of clustering tree}
    \label{fig:cluster}
\end{figure}

In the following lemma we prove that for each blow-up point $a_\alpha$ it is possible to organize the sequences $x_{i,n}\to a_\alpha$ into a finite rooted tree, which describes the clustering structure.

\begin{lemma}\label{lemclustertree}
After passing to a subsequence, for every $a_\alpha\in \{a_1,\dots, a_N\}$, the set $I(\alpha)$
admits a finite rooted tree $\mathcal T(a_\alpha)$ of subsets of $I(\alpha)$ with
the following properties.

\begin{enumerate}
\item The root is $I(\alpha)$, the leaves are the singletons $\{i\}$ with $i\in I(\alpha)$, and the
branches of every non-leaf node form a partition of that node.

\item For every node $A$,
\[
 s_{A,n}=o(S_{A,n}).
\]

\item Given a non-leaf node $A$, if $\operatorname{br}(A)=\{A_1,\dots,A_m\}$,
then
\[
 s_{A_j,n}=o(s_{A,n}), \quad S_{A_j,n}\asymp s_{A,n},
 \quad\text{for } 1\le j\le m.
\]
and
\[
 \dist_g(X_{A_j,n},X_{A_k,n})
\asymp s_{A,n},\quad \text{for }1\le j<k\le m.
\]
\end{enumerate}
\end{lemma}

\begin{proof}
Since $I(\alpha)$ is finite, after passing to a subsequence we may assume
that, for every four indices for which the denominator is nonzero, the following limit exists
\[
 \lim_{n\to +\infty}\frac{d_g(x_{i,n},x_{j,n})}
      {d_g(x_{k,n},x_{l,n})}\in [0,+\infty].
\]
Start with the root $A=I(\alpha)$. If $|A|=1$, it is a leaf. If $|A|\ge2$,
set
\[
 s_{A,n}=\diam_g X_{A,n}.
\]
Define an equivalence relation on $A$ by
\[
 i\sim_A j
 \quad\Longleftrightarrow\quad
 \frac{d_g(x_{i,n},x_{j,n})}{s_{A,n}}\to0.
\]
The existence of all distance-ratio limits shows that $\sim_A$ is indeed an
equivalence relation. Its equivalence classes are defined to be the
branches of $A$.

Distinct equivalence classes have mutual distances comparable to
$s_{A,n}$, while each equivalence class has diameter of order $o(s_{A,n})$.
Apply the same construction recursively to every nonsingleton branch.
Since each decomposition strictly decreases the cardinality, the
process terminates after finitely many steps, and its leaves are the
singletons.

For an internal branch $B$ of $A$,
\[
 s_{B,n}=\diam_g X_{B,n}=o(s_{A,n}).
\]
If $B=\{i\}$ is a leaf, separation at the bubble scale (i.e. \eqref{separ} of Proposition \ref{propextraction}) gives
\[
 r_{i,n}=o(s_{A,n}),
\]
and hence again $s_{B,n}=o(s_{A,n})$.

By definition,
\[
 S_{B,n}
 =
 \frac12\dist_g
 \bigl(X_{B,n},X_{A\setminus B,n}\bigr),
\]
and the separation of distinct branches at scale $s_{A,n}$ implies $S_{B,n}\asymp s_{A,n}$.
Consequently,
\[
 \frac{s_{B,n}}{S_{B,n}}\to0.
\]
For the root, $s_{A,n}\to0$ because $x_{i,n}\to a_{\alpha}$ for every $i\in I(\alpha)$,
whereas $S_{A,n}=S_A>0$ is fixed. This proves all the assertions.
\end{proof}

\section{Decay and vanishing on necks} \label{sec:necks}

In this section we establish Proposition \ref{propneck}, i.e. the vanishing of the mass in the necks. At a first reading one can take, for a fixed $i\in \{1,\dots,\ell\}$,  
$z_n=x_{i,n}$, $s_n=r_{i,n}$ and $S_n=S_{\{i\},n}$ to be the scale of the neck.
Then conditions \eqref{N1}-\eqref{N3} follow from Proposition \ref{propextraction}-(i), so Proposition \ref{propneck} implies
$$ \lim_{R\to\infty}\lim_{n\to\infty}
        \int_{B^g_{S_{\{i\},n}}(x_{i,n})\setminus B^g_{Rr_{i,n}}(x_{i,n})}
        \lambda_nhe^{u_n}\dV=0,$$
i.e. the mass vanishes on the neck between the bubble $x_{i,n}$ and the closest bubbles.

More generally, Proposition \ref{propneck} will handle the neck between a cluster $A$ and the ``nearby'' clusters (the other branches of $\operatorname{pred}(A)$). In this case $z_n=x_{A,n}$, $s_n=s_{A,n}$ and $S_n=S_{A,n}$, see Figure \ref{fig:cluster}. Then conditions \eqref{N2}-\eqref{N3} will follow by induction.

\begin{prop}
[Neck vanishing]\label{propneck}
Let $z_n\in M$ and
\begin{equation}\label{snSn}
        0<s_n<S_n\leq S_0,
    \qquad
    \frac{s_n}{S_n}\longrightarrow0.
\end{equation} 
Assume that  there exists some $R_0\geq 2$ such that, for every $R\geq R_0$ and
large $n$
\begin{equation}\label{N1}\tag{$\mathcal{N}_1$}
  d_n(x)\geq \frac12\,d_g(x,z_n)\quad \text{for } Rs_n
    \leq d_g(x,z_n)
    \leq S_n.
\end{equation}
Assume also that
\begin{equation}\label{N2}\tag{$\mathcal{N}_2$}
        \lim_{R\to\infty}\lim_{n\to\infty}\int_{B^g_{8Rs_n}(z_n)\setminus B^g_{Rs_n}(z_n)}\lambda_n he^{u_n}dV_g=0;
\end{equation}
and that for $R\ge R_0$ and for some $\nu_1\in (1,2)$
\begin{equation}\label{N3}\tag{$\mathcal{N}_3$}
        \liminf_{n\to\infty}\int_{B^g_{Rs_n}(z_n)}\lambda_n he^{u_n}dV_g \geq4\pi\nu_1.
\end{equation}
Then, if $S_0>0$ is sufficiently small, depending on $\nu_1$, we have
\begin{equation}\label{eqnoneck}
    \lim_{R\to\infty}\lim_{n\to\infty}
    \int_{
       B^g_{S_n}(z_n)\setminus B^g_{Rs_n}(z_n)}\lambda_n he^{u_n}dV_g=0.
\end{equation}
\end{prop}

The proof of Proposition \ref{propneck} is based on the following decay estimates of the averages on $u_n$ on geodesic circles around $z_n$.
 
\begin{lemma}
\label{lemaverage}
Let $z_n\in M$, $s_n,S_n,S_0$ be as in \eqref{snSn}. Assume there exist $R_0\geq 4$ and $\nu_1\in(1,2)$ such that \eqref{N1}, \eqref{N3} hold for every fixed $R\geq R_0$ and sufficiently large $n$. Set 
\[
    \bar u_n(t)
    :=
    \frac{1}{2\pi}
    \int_0^{2\pi}
    u_n\bigl(\exp_{z_n}(t\theta)\bigr)\,d\theta.
\]
For $R\geq R_0$, define
\begin{equation}\label{eq:delta_n-R}
    \delta_n(R)
    :=
    \int_{B^g_{8Rs_n}(z_n)\setminus B^g_{Rs_n}(z_n)}\lambda_n he^{u_n}dV_g.
\end{equation}
Then there exists $C>0$ such that, for every fixed $R\geq R_0$
and all sufficiently large $n$, one can choose
$\tau_n\in[Rs_n,2Rs_n] $
such that 
\begin{equation}\label{goodtauM2}
    \tau_n^2e^{\bar u_n(\tau_n)}
    \leq C\delta_n(R).    
\end{equation}
Moreover, for every $\nu\in(1,\nu_1)$ there exists $S_0$ sufficiently small (depending on $\nu$ and $\nu_1$) such that, for
all sufficiently large $n$, choosing $\tau_n$ as above,
one has
\begin{equation}\label{uavgdecayM2}
    \bar u_n(r)
    \leq
    \bar u_n(\tau_n)
    -2\nu\log\frac{r}{\tau_n},\quad \text{for every }2\tau_n\leq r\leq S_n\le S_0.
\end{equation}
Consequently,
\begin{equation}\label{uavgdecayMexp2}
    e^{\bar u_n(r)}
    \leq
    C\delta_n(R)\,
    \tau_n^{2\nu-2}r^{-2\nu},\quad \text{for every }2\tau_n\leq r\leq S_n.
\end{equation}
\end{lemma}

\begin{proof}
In geodesic polar coordinates centered at $z_n$, write
\begin{equation}\label{defJn}
    g=dt^2+J_n(t,\theta)^2\,d\theta^2,
    \qquad
    J_n(t,\theta)=t\bigl(1+E_n(t,\theta)\bigr),\qquad E_n(t,\theta)= O(t^2)
\end{equation}
where the error is uniform in $n$.
Since $\lambda_n\to\lambda_\infty>0$ and $h>0$,
\[
\begin{aligned}
\int_{B^g_{2Rs_n}(z_n)\setminus B^g_{Rs_n}(z_n)}
e^{u_n}\,\dV
&\leq
C\delta_n(R).
\end{aligned}
\]
For $n$ sufficiently large, $2Rs_n$ is smaller than the injectivity
radius. Moreover, $J_n(t,\theta)\geq ct$ in the range under
consideration. Jensen's inequality therefore gives
\[
\begin{aligned}
\int_{\partial B^g_t(z_n)}e^{u_n}\,d\sigma_g
&=
\int_0^{2\pi}
e^{u_n(\exp_{z_n}(t\theta))}
J_n(t,\theta)\,d\theta
\\
&\geq
ct
\int_0^{2\pi}
e^{u_n(\exp_{z_n}(t\theta))}\,d\theta\geq
ct\,e^{\bar u_n(t)}.
\end{aligned}
\]
Integrating,
\[
   \int_{Rs_n}^{2Rs_n}
    t^2 e^{\bar u_n(t)}\,\frac{dt}{t}=  \int_{Rs_n}^{2Rs_n}
    t e^{\bar u_n(t)}\,dt
    \leq C\delta_n(R).
\]
Since
\[
    \int_{Rs_n}^{2Rs_n}\frac{dt}{t}=\log 2,
\]
there exists $\tau_n\in[Rs_n,2Rs_n]$ such that \eqref{goodtauM2} holds.

\medskip

We next prove \eqref{uavgdecayM2}. Fix $\nu\in(1,\nu_1)$ and choose
\[\nu_0 =\frac{\nu+\nu_1}{2}\in (\nu,\nu_1).\]
For $R\ge R_0$, choosing $\tau_n\geq Rs_n$ as above, by assumption \eqref{N3}, 
\[
  \int_{B^g_{t}(z_n)}\lambda_n h e^{u_n} dV_g\ge   \int_{B^g_{Rs_n}(z_n)}\lambda_n h e^{u_n} dV_g \geq4\pi\nu_0,\quad
    \text{for every }t\geq\tau_n,\,n\text{ large}.
\]
Let
\begin{equation}\label{condsr}
    \tau_n\leq s<r\leq S_n,
    \qquad
    \frac{r}{s}\geq2,
\end{equation}
and define
\[
    \varphi_{s,r}(x)
    :=
    \begin{cases}
    1,
    & d_g(x,z_n)\leq s,
    \\[2mm]
    \displaystyle
    \frac{\log r-\log d_g(x,z_n)}
         {\log r-\log s},
    & s<d_g(x,z_n)<r,
    \\[3mm]
    0,
    & d_g(x,z_n)\geq r.
    \end{cases}
\]
Then
\[
    0\leq\varphi_{s,r}\leq1,
    \qquad
    \operatorname{supp}\varphi_{s,r}
    \subset B^g_r(z_n),
\]
and
\[
    \int_M|\nabla_g\varphi_{s,r}|^2\,\dV
    \leq
    \frac{C}{\log(r/s)}
    \leq C.
\]
Consequently,
\[
    \|\varphi_{s,r}\|_{H^1(M)}\leq C,\quad \text{uniformly for }s,r\text{ as in \eqref{condsr}.}
\]
Testing \eqref{PSeq} with  $\varphi_{s,r}$ gives
\[
\begin{aligned}
\int_M
\langle\nabla_g u_n,\nabla_g\varphi_{s,r}\rangle\,\dV
&=
\int_M
\lambda_nhe^{u_n}\varphi_{s,r}\,\dV
-
\lambda_n\int_M\varphi_{s,r}\,\dV
+o(1),
\end{aligned}
\]
where $o(1)\to0$ uniformly in $s$ and $r$ as in \eqref{condsr}. Since
$\varphi_{s,r}=1$ on $B^g_s(z_n)$,
\[
    \int_M
    \lambda_nhe^{u_n}\varphi_{s,r}\,\dV
    \geq 4\pi\nu_0.
\]
Moreover,
\[
    \lambda_n\int_M\varphi_{s,r}\,\dV
    \leq Cr^2.
\]
It follows that
\begin{equation}\label{nablaunablaphi}
\int_M \langle\nabla_g u_n,\nabla_g\varphi_{s,r}\rangle\,\dV \geq
    4\pi\nu_0-Cr^2+o(1).
\end{equation}
On the annulus $A^g_{s,r}(z_n) :=
    B^g_{r}(z_n)\setminus B^g_s(z_n)$,
\[
    \partial_t\varphi_{s,r}
    =
    -\frac{1}{t\log(r/s)}.
\]
Therefore, using \eqref{defJn}, and setting $w_n:=u_n-v_n$
\[
\begin{aligned}
\int_M
\langle\nabla_g u_n,\nabla_g\varphi_{s,r}\rangle\,\dV
&=
-\frac{1}{\log(r/s)}
\int_s^r\int_0^{2\pi}
\partial_tu_n(t,\theta)
\frac{J_n(t,\theta)}{t}
\,d\theta\,dt.\\
&=-\frac{1}{\log(r/s)}
\int_s^r\int_0^{2\pi}
\partial_tu_n(t,\theta)
\,d\theta\,dt.\\
&\quad -\frac{1}{\log(r/s)}
\int_s^r\int_0^{2\pi}
\partial_t(v_n+w_n)(t,\theta)
E_n(t,\theta)
\,d\theta\,dt.
\end{aligned}
\]
The first integral on the right-hand side gives
\[
    \frac{2\pi}{\log(r/s)}
    \bigl(\bar u_n(s)-\bar u_n(r)\bigr).
\]
By \eqref{N1}, every $x\in A^g_{s,r}(z_n)$ satisfies $d_n(x)\geq 
    \frac12 \,d_g(x,z_n)$, hence
Proposition~\ref{propgradient} gives
\[
    |\nabla_gv_n(x)|
    \leq
    \frac{C}{d_g(x,z_n)}\quad \text{for }x\in A^g_{s,r}(z_n).
\]
It follows that
\[
\begin{aligned}
\frac{1}{\log(r/s)}
\left|
\int_s^r\int_0^{2\pi}
\partial_tv_n(t,\theta)E_n(t,\theta)
\,d\theta\,dt
\right|
&\leq
\frac{C}{\log(r/s)}
\int_s^r t\,dt
=O(r^2).
\end{aligned}
\]
For the term involving $w_n$, the Cauchy--Schwarz inequality and
$dV_g\simeq t\,dt\,d\theta$ give
\[
\begin{aligned}
&\frac{1}{\log(r/s)}
\left|
\int_s^r\int_0^{2\pi}
\partial_tw_n(t,\theta)E_n(t,\theta)
\,d\theta\,dt
\right|
\leq
Cr^2\|\nabla_gw_n\|_{L^2(M)}
=o(r^2),
\end{aligned}
\]
uniformly in the admissible radii. Together with \eqref{nablaunablaphi}, this gives
\[
\begin{aligned}
4\pi \nu_0 +o(1)+O(r^2)\le \int_M
\langle\nabla_g u_n,\nabla_g\varphi_{s,r}\rangle\,\dV
&=
\frac{2\pi}{\log(r/s)}
\bigl(\bar u_n(s)-\bar u_n(r)\bigr)
+O(r^2),
\end{aligned}
\]
i.e.
\[
    \frac{1}{\log(r/s)}
    \bigl(\bar u_n(s)-\bar u_n(r)\bigr)
    \geq 2\nu_0 +o(1)+O(r^2),
\]
with $o(1)\to 0$ and $O(r^2)\le Cr^2$ uniformly for $r,s$ as in \eqref{condsr}.
Since $r\leq S_n\leq S_0$,
we may choose $S_0$ sufficiently small depending on $\nu$ and $\nu_1$\footnote{this is, of course, not necessary if $S_n\to 0$.}  such that, for every
$R\geq R_0$ and all large $n$, the preceding estimates
imply
\[
    \frac{2\pi}{\log(r/s)}
    \bigl(\bar u_n(s)-\bar u_n(r)\bigr)
    \geq4\pi\nu.
\]
Hence
\[
    \bar u_n(r)
    \leq
    \bar u_n(s)
    -
    2\nu\log\frac{r}{s}.
\]
Taking $s=\tau_n$, we obtain \eqref{uavgdecayM2}.

 Finally, \eqref{uavgdecayMexp2} follows immediately from \eqref{goodtauM2} and \eqref{uavgdecayM2}.
\end{proof}

\noindent\emph{Proof of Proposition \ref{propneck}.}
Fix \[\nu=\frac{1+\nu_1}{2}\in(1,\nu_1),\]
and pick $S_0>0$ as given by Lemma \ref{lemaverage}. For every sufficiently large $R$ and all sufficiently large $n$, depending on $R$, Lemma~\ref{lemaverage} provides
$\tau_n\in[Rs_n,2Rs_n]$ such that \eqref{goodtauM2}, \eqref{uavgdecayMexp2} hold.

We claim that
\begin{equation}\label{intAr2r}
    \int_{A^g_{r,2r}(z_n)}
    \lambda_n h e^{u_n}\,dV_g
    \leq
    C\delta_n(R)
    \left(\frac{\tau_n}{r}\right)^{2\nu-2},
    \qquad\text{for }
    4\tau_n\leq r\leq\frac{S_n}{2},
\end{equation}
where
$
A^g_{r,2r}(z_n)
    :=
    B^g_{2r}(z_n)\setminus B^g_r(z_n).$
Define
the rescaled functions
\[
    \widehat u_{n,r}(y)
    :=
    u_n\bigl(\exp_{z_n}(ry)\bigr)+2\log r,
    \qquad y\in D:=A_{1,2}(0)=B_2(0)\setminus B_1(0).
\]
Let also
\[
    \widehat g_{n,r}
    :=
    r^{-2}
    \bigl(\exp_{z_n}\circ(r\,\cdot)\bigr)^*g.
\]
For $y\in D$ one has
\[
    r
    \leq
    d_g\bigl(\exp_{z_n}(ry),z_n\bigr)
    \leq
    2r.
\]
By our choice of $r$ in \eqref{intAr2r}, assumption \eqref{N1} gives
\[
d_n(x) \geq \frac12 d_g(x,z_n) \geq \frac r2,
    \qquad \text{for }x\in A^g_{r,2r}(z_n).
\]
Proposition~\ref{propgradient} therefore yields
\[
    |\nabla_gv_n|
    \leq\frac{C}{r}
    \qquad
    \text{on }A^g_{r,2r}(z_n).
\]
Set $w_n:=u_n-v_n\to0$ in $H^1(M)$. By the conformal invariance of the Dirichlet integral,
\begin{equation}\label{stimahatun}
    \int_D
    |\nabla_{\widehat g_{n,r}}\widehat u_{n,r}|^2
    \,dV_{\widehat g_{n,r}}
    \leq
    C\int_{A^g_{r,2r}(z_n)}
    |\nabla_gv_n|^2\,dV_g
    +
    C\int_M|\nabla_gw_n|^2\,dV_g \leq C.
\end{equation}
The rescaled metrics $\widehat g_{n,r}$ are uniformly equivalent to the Euclidean metric on $D$. Indeed, $2r\leq S_n\leq S_0$,
where $S_0$ is chosen smaller than the injectivity radius. Thus the family
$\widehat g_{n,r}$ is uniformly elliptic and uniformly bounded in $C^1(D)$.

Define
\[
m_{n,r}:=\fint_D \widehat u_{n,r}\,dy.
\]
By the uniform equivalence of $\widehat g_{n,r}$ with the Euclidean
metric and \eqref{stimahatun},
\[
\int_D |\nabla \widehat u_{n,r}|^2\,dy \le C.
\]
Hence, by the Poincar\'e inequality on the fixed annulus $D$, $\|\widehat u_{n,r}-m_{n,r}\|_{H^1(D)}\le C$.
By the Moser--Trudinger inequality \eqref{MT2} and the equivalence of the metrics, 
\begin{equation}\label{eq:reverse-jensen}
    \int_D
\exp\bigl(|\widehat u_{n,r}-m_{n,r}|\bigr)
\,dV_{\widehat g_{n,r}}
\le C\int_D
\exp\bigl(|\widehat u_{n,r}-m_{n,r}|\bigr)\,dy
\le C.
\end{equation}
We next estimate $m_{n,r}$. Since $D=A_{1,2}(0)$ and
$|D|=3\pi$, using polar coordinates $y=\rho \, \theta$ we have the exact
identity
\[
\begin{aligned}
m_{n,r}
&=
2\log r
+\frac{1}{3\pi}
\int_1^2\int_0^{2\pi}
u_n\bigl(\exp_{z_n}(r\rho\theta)\bigr)
\,\rho\,d\theta\,d\rho
\\
&=
2\log r+\frac{2}{3}\int_1^2
\rho\,\bar u_n(r\rho)\,d\rho.
\end{aligned}
\]
For $\rho\in[1,2]$, our choice
$4\tau_n\le r\le S_n/2$ gives $2\tau_n\le r\rho\le S_n$.
Hence \eqref{uavgdecayMexp2} yields
\[
\bar u_n(r\rho)
\le
\log\bigl(C\delta_n(R)\bigr)
+(2\nu-2)\log\tau_n
-2\nu\log(r\rho)
\]
and therefore, since $\rho\ge1$ and
$\frac23\int_1^2\rho\,d\rho=1$,
\[
\begin{aligned}
m_{n,r}
&\le
2\log r
+\log\bigl(C\delta_n(R)\bigr)
+(2\nu-2)\log\tau_n
-2\nu\log r
\\
&=
\log\bigl(C\delta_n(R)\bigr)
+(2\nu-2)\log\frac{\tau_n}{r}.
\end{aligned}
\]
Consequently,
\[
e^{m_{n,r}}
\le
C\delta_n(R)
\left(\frac{\tau_n}{r}\right)^{2\nu-2}.
\]
By the change of variables $x=\exp_{z_n}(ry)$, using \eqref{eq:reverse-jensen} and the above estimate we get
\[
\begin{aligned}
    \int_{A^g_{r,2r}(z_n)}
    \lambda_nhe^{u_n}\,dV_g
    &\leq
    C\int_{A^g_{r,2r}(z_n)}
    e^{u_n}\,dV_g
    =
    C\int_D
    e^{\widehat u_{n,r}}\,dV_{\widehat g_{n,r}}
    \\
    &\leq
    Ce^{m_{n,r}}
    \leq
    C\delta_n(R)
    \left(\frac{\tau_n}{r}\right)^{2\nu-2}.
\end{aligned}
\]
This proves \eqref{intAr2r}.

We now sum over dyadic annuli. Since $s_n/S_n\to0$ and
$\tau_n\leq2Rs_n$, for every fixed $R$ and all sufficiently large $n$, $4\tau_n
    \leq
    8Rs_n
    \leq
    S_n/2$.
Set
\[
    r_{k,n}:=2^k4\tau_n,
    \qquad k\geq0,
\]
and let $K_n$ be the largest integer such that
$r_{K_n,n}\leq {S_n}/{2}$.
By maximality, $2r_{K_n,n}>{S_n}/{2}$. Hence 
\[
    B^g_{S_n}(z_n)\setminus B^g_{4\tau_n}(z_n)
    \subset
    \bigcup_{k=0}^{K_n}
    A^g_{r_{k,n},2r_{k,n}}(z_n)
    \,\cup\,
    A^g_{S_n/2,S_n}(z_n).
\]
Applying \eqref{intAr2r} to the dyadic annuli and to the final annulus with $r=S_n/2$, we obtain
\[
\begin{aligned}
    &\int_{B^g_{S_n}(z_n)\setminus B^g_{4\tau_n}(z_n)}
    \lambda_nhe^{u_n}\,dV_g
    \\
    &\qquad\leq
    C\delta_n(R)
    \sum_{k=0}^{K_n}
    \left(
        \frac{\tau_n}{2^k4\tau_n}
    \right)^{2\nu-2}
    +
    C\delta_n(R)
    \left(
        \frac{2\tau_n}{S_n}
    \right)^{2\nu-2}
    \\
    &\qquad\leq
    C_\nu\delta_n(R),
\end{aligned}
\]
because $\nu>1$.

It remains to estimate the inner part. Since
$\tau_n\in[Rs_n,2Rs_n]$, one has
\[
    B^g_{4\tau_n}(z_n)\setminus B^g_{Rs_n}(z_n)
    \subset
    B^g_{8Rs_n}(z_n)\setminus B^g_{Rs_n}(z_n),
\]
and hence, by the definition \eqref{eq:delta_n-R} of $\delta_n(R)$,
\[
    \int_{B^g_{4\tau_n}(z_n)\setminus B^g_{Rs_n}(z_n)}
    \lambda_nhe^{u_n}\,dV_g
    \leq
    \delta_n(R).
\]
Combining the inner and outer estimates gives
\[
    \int_{B^g_{S_n}(z_n)\setminus B^g_{Rs_n}(z_n)}
    \lambda_nhe^{u_n}\,dV_g
    \leq
    C_\nu\delta_n(R),
\]
for every sufficiently large $R$ and all sufficiently large $n$ depending on $R$. Thus
\[
    \limsup_{n\to\infty}
    \int_{B^g_{S_n}(z_n)\setminus B^g_{Rs_n}(z_n)}
    \lambda_nhe^{u_n}\,dV_g
    \leq
    C_\nu
    \limsup_{n\to\infty}\delta_n(R).
\]
Then \eqref{eqnoneck} follows directly from \eqref{N2}.
\hfill $\square$

\section{Quantization} \label{sec:cluster}

In this section we analyze the clustering behavior of bubbles and use the results of Sections \ref{sec:extract} and \ref{sec:necks} to prove Theorem \ref{trmmain}. The notation of Section \ref{sec:trees} is employed throughout, and we will write $d_g$ in place of $\dist_g$, also for the distance between subsets of $M$.

In Proposition \ref{propclusterclearing} below we show that no mass concentrates in the intermediate regions between one cluster $A$ and its branches, ``at scales $s_{A,n}$''. Its proof uses ideas from the seminal work of Brezis-Merle, see \cite[Theorem 3]{brezis-merle-1991}: after a suitable scaling,  either the sequence is compact or it converges to $-\infty$ locally uniformly away from the blow-up points. The compact alternative is ruled out using the finiteness of the mass and the presence of a concentration of at least $4\pi$ to obtain a contradiction. Similar ideas, but without the rescaling, are also used in Proposition \ref{nodiffuse}.
\begin{prop}
\label{propclusterclearing}
Let $A$ be a non-leaf node with branches
$A_1,\dots,A_m$. 
Then, for every (small) $\delta>0$ one has
\[
 \int_{B^g_{s_{A,n}/\delta}(x_{A,n})\setminus \bigcup_{j=1}^m B^g_{\delta s_{A,n}}(x_{A_j,n})} \lambda_nhe^{u_n}\,dV_g\to0.
\]
\end{prop}

\begin{proof} We proceed by steps.

\medskip
\noindent
\emph{Step 1.}
Set
\begin{equation}\label{defFn}
     F_n(y):=\exp_{x_{A,n}}(s_{A,n}y),
 \qquad
 \widehat g_n:=s_{A,n}^{-2}F_n^*g.
\end{equation}
For each branch $A_j$ of $A$ with representative point
$x_{A_j,n}\in X_{A_j,n}$, set
\[
y_{j,n}
:=
s_{A,n}^{-1}\exp_{x_{A,n}}^{-1}(x_{A_j,n}),
\qquad j=1,\dots,m.
\]
Since $x_{A,n},x_{A_j,n}\in X_{A,n}$ and
\[
s_{A,n}=\diam_g X_{A,n},
\]
we have
\[
|y_{j,n}|
=
\frac{d_g(x_{A,n},x_{A_j,n})}{s_{A,n}}
\le 1.
\]
After passing to a subsequence,
there exist points $y_1,\dots,y_m\in\overline{B}_1(0)\subset\R^2$ such that
\begin{equation}\label{eq:branch-limits}
y_{j,n}\to y_j,
\qquad j=1,\dots,m.
\end{equation}

We claim that these points are pairwise distinct. Indeed, for $j\neq k$,
the separation property of distinct branches (see Lemma \ref{lemclustertree}) gives $d_g(X_{A_j,n},X_{A_k,n})\asymp  s_{A,n}$.
In particular, there exists $c_A>0$ such that
\[
\liminf_{n\to +\infty}\frac{d_g(x_{A_j,n},x_{A_k,n})}{s_{A,n}}\ge c_A,\quad \text{for }1\le j< k\le m.
\]
On the other hand, with $F_n$ and $\widehat g_n$ as in \eqref{defFn},
we have $\widehat g_n\to g_{\R^2}$ in  $C^2_{\loc}(\R^2)$.
As $x_{A_j,n}=F_n(y_{j,n})$ and
$x_{A_k,n}=F_n(y_{k,n})$,  local convergence of the rescaled
metrics implies
\[
|y_j-y_k|
=
\lim_{n\to\infty}
\frac{d_g(x_{A_j,n},x_{A_k,n})}{s_{A,n}}
\ge c_A>0, \quad \text{for }1\le j< k\le m.
\]
Thus $y_1,\dots,y_m$ are pairwise distinct.

Moreover, if $i\in A_j$ (i.e. $x_{i,n}$ belongs to the cluster $A_j$), then
\[
\frac{d_g(x_{i,n},x_{A_j,n})}{s_{A,n}}
\le
\frac{\diam_g X_{A_j,n}}{s_{A,n}}
\to0.
\]
This proves 
\begin{equation}\label{eq:centres-branch-limits}
 s_{A,n}^{-1}
 \exp_{x_{A,n}}^{-1}(x_{i,n})
 \to y_j,\quad \text{for every }i\in A_j.
\end{equation}
By contrast, the centers not belonging to $A$ (if any) escape to infinity at
the scale $s_{A,n}$, since $s_{A,n}=o(S_{A,n})$ (see Lemma \ref{lemclustertree}),
so that
\begin{equation}\label{eq:centres-infinity}
 \frac{d_g(x_{i,n},x_{A;n})}{s_{A,n}}\to+\infty,\quad \text{for every }i\in \{1,\dots,\ell\}\setminus A.
\end{equation}
Set now
\[
Y:=\{y_1,\dots,y_m\}.
\]

\medskip
\noindent
\emph{Step 2.}
Fix an open set
\[
\Omega\Subset\mathbb R^2\setminus Y.
\]
By the definition of $Y$ and since the centers $x_{i,n}$ for $i\not\in A$ are at a distance at least $S_{A,n}$ from the centers $x_{j,n}$ with $j\in A$, there exists $c_\Omega>0$ such that
\begin{equation}\label{eq:distance-centres}
d_n(F_n(z))\ge c_\Omega s_{A,n}
\end{equation}
for every $z\in\Omega$ and sufficiently large $n$, where $d_n$ is as in \eqref{defdn}.

Moreover, by local uniform convergence of $\widehat g_n\to g_{\mathbb R^2}$, 
\begin{align}\label{eq:ball-inclusion}
F_n\bigl(B_\rho(z)\bigr)
\subset
B^g_{2\rho s_{A,n}}\bigl(F_n(z)\bigr)
\end{align}
whenever $z\in\Omega$, $\rho\le 1$ and $n\geq n_\Omega$ is sufficiently large. Choose $1>\rho_\Omega>0$ so that
\[
2\rho_\Omega\le \frac{c_\Omega}{C_*},
\]
where $C_*$ is as in Proposition \ref{propextraction}.
Then, by \eqref{eq:distance-centres},
\[
2\rho_\Omega s_{A,n}
\le \frac{d_n(F_n(z))}{C_*},
\]
so that \eqref{nofurther} holds with $x=F_n(z)$ and $r=2\rho_\Omega s_{A,n}$.
Now set
\[
U_n(y):=u_n(F_n(y))+2\log s_{A,n},
\qquad
V_n(y):=v_n(F_n(y))+2\log s_{A,n},
\]
so that
\begin{equation}\label{eq:cluster-rescaled}
-\Delta_{\widehat g_n}V_n
=
f_n(y)e^{U_n(y)}-\lambda_ns_{A,n}^2,
\qquad
f_n(y):=\lambda_n h(F_n(y)).
\end{equation}
For a fixed $p>2$ and after scaling, \eqref{nofurther}, \eqref{eq:ball-inclusion} give
\[
\left(
\int_{B_{\rho_\Omega}(z)}
e^{pU_n}\,dV_{\widehat g_n}
\right)^{1/p}
\le C_g \rho_*(2\rho_\Omega)^{\frac{2}{p}-2} =: C_\Omega.
\]
The constant is uniform for $z\in\Omega$. Up to enlarging $C_\Omega$, a finite covering 
gives
\begin{equation}\label{eq:cluster-local-Lp}
\|e^{U_n}\|_{L^p(\Omega,dV_{\widehat g_n})}
\le C_\Omega.
\end{equation}
In particular, since $s_{A,n}=o(1)$,
\begin{equation}\label{eq:cluster-source-Lp}
\|f_ne^{U_n}-\lambda_ns_{A,n}^2\|_
 {L^p(\Omega,dV_{\widehat g_n})}
\le C_\Omega.
\end{equation}

\medskip
\noindent
\emph{Step 3.}
Choose a nonempty reference ball $B_0\Subset\R^2\setminus Y$
and constants $c_n$ such that
\[
 \fint_{B_0}\bigl(U_n-(V_n+c_n)\bigr)\,dy=0.
\]
Set
\[
 \widetilde V_n:=V_n+c_n.
\]
Since $\|\nabla_g(u_n-v_n)\|_{L^2(M)}\to0$,
conformal invariance of the Dirichlet integral and Poincaré's
inequality on connected bounded subsets of $\R^2\setminus Y$ give
\begin{equation}\label{eq:H1-cluster-close}
 U_n-\widetilde V_n\to0
 \quad\text{in }H^1_{\mathrm{loc}}(\R^2\setminus Y),
\end{equation}
(this holds more generally in $H^1_{\loc}(\R^2)$).
Here the same constants $c_n$ work on every compact subset because
$\R^2\setminus Y$ is connected and one can connect different  compact sets to $B_0$
by finitely many overlapping balls and apply Poincaré's inequality
on their union.

We now prove the following alternative:
\begin{enumerate}
\item $\widetilde V_n$ is bounded in
      $C^{1,\alpha}_{\mathrm{loc}}(\R^2\setminus Y)$; or
\item $\widetilde V_n\to-\infty$ locally uniformly on
      $\R^2\setminus Y$.
\end{enumerate}

Let $\Omega'\Subset\Omega\Subset\R^2\setminus Y$
be smooth and connected. Let $W_n$ solve
\[
 \begin{cases}
 -\Delta_{\widehat g_n}W_n
 =f_ne^{U_n}-\lambda_ns_{A,n}^2&\text{in }\Omega,\\
 W_n=0&\text{on }\partial\Omega,
 \end{cases}
\]
and set
\[
 Z_n:=\widetilde V_n-W_n.
\]
By \eqref{eq:cluster-source-Lp} and the uniform ellipticity of
$\widehat g_n$, the Dirichlet estimates give
$\|W_n\|_{W^{2,p}(\Omega)}
\le C_\Omega$.
Since $p>2$, for every $\alpha\le 1-\frac2p$,
\[
\|W_n\|_{C^{1,\alpha}(\overline\Omega)}
\le C_\Omega.
\]

Furthermore,
\[
 \int_{\Omega'}U_n^+\,dy
 \le \int_{\Omega'}e^{U_n}\,dy
 \le C,
\]
and \eqref{eq:H1-cluster-close} implies
\[
 \int_{\Omega'}\widetilde V_n^+\,dy\le C.
\]
It follows that
\[
 \int_{\Omega'}Z_n^+\,dy\le C.
\]

We now proceed similarly to the proof of
Proposition~\ref{propextraction}. Fix a ball
\[
B_{2r}(z)\Subset\Omega'\Subset\Omega\Subset\R^2\setminus Y.
\]
Since $Z_n$ is $\widehat g_n$-harmonic and
$\widehat g_n\to g_{\R^2}$ in $C^2_{\loc}$, the local estimates for
nonnegative subsolutions, with constants uniform in $n$, yield
\[
\sup_{B_{3r/2}(z)}Z_n
\le
\sup_{B_{3r/2}(z)}Z_n^+
\le
C\int_{B_{2r}(z)}Z_n^+\,dy
\le T,
\]
see e.g.
\cite[Theorems~8.17]{gilbargtrudinger}.
Set
\[
\beta_n:=\|Z_n\|_{L^1(B_r(z))}.
\]
After passing to a subsequence, there are two possibilities.

\medskip

\noindent\textbf{Case 1: $\beta_n\le C$.} If $(\beta_n)$ is bounded, the interior estimates for
$\widehat g_n$-harmonic functions give
\[
\|Z_n\|_{C^{1,\alpha}(B_{r/2}(z))}\le C, \quad \text{hence also} \quad \|\widetilde V_n\|_{C^{1,\alpha}(B_{r/2}(z))}\le C.
\]

\medskip

\noindent\textbf{Case 2: $\beta_n\to+\infty$.}
Define $H_n:=T+1-Z_n$.
Then $H_n$ is positive and $\widehat g_n$-harmonic in
$B_{3r/2}(z)$. Since $Z_n^+$ is bounded in $L^1(B_r(z))$, the condition
$\beta_n\to+\infty$ implies
\[
\int_{B_r(z)}H_n\,dy\to+\infty.
\]
The uniform Harnack inequality gives
\[
\sup_{B_r(z)}H_n
\le C_H\inf_{B_r(z)}H_n.
\]
Consequently, $\inf_{B_r(z)}H_n\to+\infty$,
and therefore
\[
Z_n\to-\infty
\qquad\text{uniformly on }B_r(z),
\]
hence
$$\widetilde V_n \to-\infty
\qquad\text{uniformly on }B_r(z).$$

\medskip

The alternatives of Case 1 and Case 2 are incompatible
on overlapping balls. Since every ball 
 $B_{2r}\Subset \R^2\setminus Y$ is contained in some $\Omega'\Subset \Omega$, the above applies to all such balls. Then, since $\R^2\setminus Y$ is connected, the same alternative holds
throughout, hence after a diagonal extraction, either
\[
(\widetilde V_n)\ \text{is bounded in }
C^{1,\alpha}_{\loc}(\R^2\setminus Y),
\]
or
\begin{equation}\label{Vn-infty}
    \widetilde V_n\to-\infty
\qquad\text{locally uniformly on }\R^2\setminus Y.
\end{equation}

\medskip
\noindent
\emph{Step 4.}
Assume that the first alternative holds. Fix any $j$ with $1\le j\le m$,
and choose $r>0$ so
small that $B_{2r}(y_j)\cap Y=\{y_j\}$.
Now fix any $i\in A_j$. Proposition \ref{propextraction} and $r_{i,n}=o(s_{A,n})$ imply that the rescaled measures
\[
 \mu_n:=f_ne^{U_n}\,dV_{\widehat g_n}
\]
satisfy, after passing to a subsequence,
\begin{equation}\label{mu8pi}
 \mu_n\rightharpoonup\mu,
 \qquad
 \mu(\{y_j\})\ge 8\pi.
\end{equation}
Case 1 implies
\[
 \|\widetilde V_n\|_{L^\infty(\partial B_r(y_j))}\le C_r.
\]
Let $\zeta_n$ be the solution of
\[
 \begin{cases}
 -\Delta_{\widehat g_n}\zeta_n
 =f_ne^{U_n}-\lambda_ns_{A,n}^2
       &\text{in }B_r(y_j),\\
 \zeta_n=-C_r&\text{on }\partial B_r(y_j).
 \end{cases}
\]
Then $\widetilde V_n\ge -C_r=\zeta_n$ on $\partial B_r(y_j)$.
Since $\Delta_{\widehat g_n} (\widetilde V_n-\zeta_n)=0$,
the maximum principle gives
\begin{equation}\label{eq:comparison-cluster}
 \widetilde V_n\ge\zeta_n
 \quad\text{in }B_r(y_j).
\end{equation}

By elliptic estimates (see \cite{stampacchia1965egularity}), $(\zeta_n)$ is bounded in $W^{1,s}$ for every $s<2$.
After passing to a subsequence it converges almost everywhere and
locally uniformly away from $y_j$ to $\zeta$ solving
\[
 -\Delta\zeta=\mu
 \quad\text{in }B_r(y_j),
 \qquad
 \zeta=-C_r\quad\text{on }\partial B_r(y_j).
\]
Then \eqref{mu8pi} gives
\begin{equation}\label{eq:log-lower-cluster}
 \zeta(y)
 \ge
 4\log\frac1{|y-y_j|}-C
 \qquad\text{near }y_j.
\end{equation}
Fix $0<\eta<r/2$. On the fixed annulus
 $A_{\eta,r/2}(y_j)$,
we have in Case 1 a uniform
$C^1$ bound for $\widetilde V_n$. By
\eqref{eq:H1-cluster-close} and Lemma \ref{exppqmean},
\[
 \int_{A_{\eta,r/2}(y_j)}
 \big|e^{U_n}-e^{\widetilde V_n}\big|\,dy
 \to0.
\]
Since, for $R>2$ and $n$ large,
\[
 \int_{B_R(0)}e^{U_n}\,dV_{\widehat g_n}
 \leq
 \int_M e^{u_n}\,dV_g
 \le C,
\]
we obtain, using \eqref{eq:comparison-cluster} and Fatou's lemma,
\[
 \int_{A_{\eta,r/2}(y_j)}e^\zeta\,dy\le C,
\]
where $C$ is independent of $\eta$. Letting $\eta\downarrow0$
contradicts \eqref{eq:log-lower-cluster}, because for some $\delta>0$
\[
 e^\zeta
 \ge \frac{\delta}{|\cdot -y_j|^{4}}
 \notin L^1(B_r(y_j)).
\]
Thus Case 1 is impossible.

\medskip
\noindent
\emph{Step 5.}
We have proven \eqref{Vn-infty}.
In particular, for a prescribed compact set $K\Subset \R^2\setminus Y$ and every $q>1$,
\[
 \int_K e^{q\widetilde V_n}\,dy\to0.
\]
Combining this with
\eqref{eq:H1-cluster-close} and Lemma \ref{exppqmean} gives
\[
 \int_K e^{U_n}\,dV_{\widehat g_n}\to0,
\]
and therefore
\[
 \int_{F_n(K)}\lambda_nhe^{u_n}\,dV_g
 =
 \int_K f_ne^{U_n}\,dV_{\widehat g_n}
 \to0.
\]
Since, for every $\delta>0$ there exists a compact set $K\Subset \R^2\setminus Y$ such that
$$B^g_{s_{A,n}/\delta}(x_{A,n})\setminus \bigcup_{j=1}^mB^g_{\delta s_{A,n}}(x_{A_j,n})\subset F_n(K),$$
the proposition is proven.
\end{proof}

Next, we inductively prove that around each cluster $A$, at scale $s_{A,n}$ a mass of $8\pi|A|$ concentrates, while on the necks, up to distance $S_{A,n}$ no further mass concentrates.
\begin{prop}\label{propclusterquant}
Let $A$ be any node of the tree $\mathcal T(a_\alpha)$. Then
\begin{equation}\label{clustercoremass}\tag{$\mathcal{Q}(A)$}
 \lim_{R\to\infty}\lim_{n\to\infty}
 \int_{B^g_{Rs_{A,n}}(x_{A,n})}\lambda_nhe^{u_n}\dV
 =8\pi|A|,
\end{equation}
and
\begin{equation}\label{clusterneck}\tag{$\mathcal{N}(A)$}
 \lim_{R\to\infty}\limsup_{n\to\infty}
 \int_{B^g_{S_{A,n}}(x_{A,n})\setminus B^g_{Rs_{A,n}}(x_{A,n})}
 \lambda_nhe^{u_n}\dV=0.
\end{equation}
\end{prop}

\begin{proof}
We argue by induction on the height of the node, from the leaves to the root.

If $A=\{i\}$ is a leaf, then $x_{A,n}=x_{i,n}$,
$s_{A,n}=r_{i,n}$,
so that \eqref{clustercoremass} follows at once from Proposition \ref{propextraction}-(i), while \eqref{clusterneck} follows from Proposition \ref{propneck} applied with $z_n=x_{i,n}$, $s_n=r_{i,n}$, $S_n=S_{\{i\},n}$ (its assumptions are all verified: \eqref{N1} is trivial by construction, while \eqref{N2}, \eqref{N3} follow  from \eqref{clustercoremass}).

Let now $A$ be a non-leaf node with branches
$A_1,\ldots,A_m$, and assume that $(\mathcal Q(A_j))$ and
$(\mathcal N(A_j))$ hold for every $j=1,\ldots,m$.
Then \eqref{clustercoremass} follows at once from this inductive hypothesis and Proposition \ref{propclusterclearing}, together with $S_{A_j,n}\asymp s_{A,n}$ for $1\le j\le m$.

It remains to prove \eqref{clusterneck}. We want to apply Proposition~\ref{propneck} with $z_n=x_{A,n}$, $s_n=s_{A,n}$ and $S_n=S_{A,n}$. Set
\[
    \delta_{A,n}(R)
    :=
    \int_{
        B^g_{8Rs_{A,n}}(x_{A,n})
        \setminus
        B^g_{Rs_{A,n}}(x_{A,n})
    }
    \lambda_nhe^{u_n}\,dV_g.
\]
Property $\eqref{clustercoremass}$ gives
\[
    \lim_{R\to\infty}
    \limsup_{n\to\infty}
    \delta_{A,n}(R)
    =0,
\]
hence condition \eqref{N2} is satisfied.

We next verify \eqref{N1}. Let
\[
    Rs_{A,n}
    \leq
    t:=d_g(x,x_{A,n})
    \leq
    S_{A,n}.
\]
Since $x_{A,n}\in X_{A,n}$ and
$\operatorname{diam}_gX_{A,n}= s_{A,n}$, for every $i\in A$,
\[
    d_g(x,x_{i,n})
    \geq
    t-s_{A,n}
    \geq
    \left(1-\frac1R\right)t.
\]
Moreover, by the definition of $S_{A,n}$ and the separation of the
different branches of the tree, every center outside $A$ is, for all
large $n$, at distance at least $2S_{A,n}$ from $x_{A,n}$. Therefore,
for every $k\notin A$,
\[
    d_g(x,x_{k,n})
    \geq
    2S_{A,n}-t
    \geq
    S_{A,n}
    \geq t.
\]
Thus, for $R\geq2$,
\[
    d_n(x)
    \geq
    \frac12d_g(x,x_{A,n})
\quad \text{whenever} \quad
    Rs_{A,n}
    \leq d_g(x,x_{A,n})
    \leq S_{A,n}.
\]
Condition \eqref{N1} is therefore satisfied. Condition \eqref{N3} follows immediately from \eqref{clustercoremass}. Then, applying Proposition~\ref{propneck},
we obtain \eqref{clusterneck} and complete the induction.
\end{proof}

Finally, we prove that locally away from each blow-up point $a_\alpha$, no mass concentrates.
\begin{prop}\label{nodiffuse}
For every $\delta>0$ we have
\begin{equation}\label{nodiffuseeq}
\limsup_{n\to\infty}
\int_{M\setminus\bigcup_{\alpha=1}^NB^g_\delta(a_\alpha)}
\lambda_nhe^{u_n}\dV=0.
\end{equation}
\end{prop}

\begin{proof}
Set
\[
\mu_n:=\lambda_n he^{u_n}\,dV_g,\qquad \mathcal{A}:=\{a_1,\dots,a_N\},\qquad
m_\alpha:=\#\{i:\ x^{(i)}=a_\alpha\},
\]
and, since $\mu_n(M)=\lambda_n\le\sup_n\lambda_n<+\infty$, assume after passing to a subsequence that
$\mu_n\rightharpoonup\mu$ weakly as measures on $M$. 

The proof follows the scheme of Proposition \ref{propclusterclearing}, with $M\setminus \mathcal{A}$ replacing $\mathbb R^2\setminus Y$, $g$ replacing $\widehat g_n$ and $(u_n,v_n)$ replacing $(U_n,\widetilde V_n)$. By Lemma \ref{lemclose}, $u_n-v_n\to 0$ in $H^1(M)$
so, unlike in Proposition \ref{propclusterclearing}, no normalizing constants $c_n$ are needed here. Moreover we will use Lemma \ref{exppqM} instead of Lemma \ref{exppqmean}.

\medskip
\noindent\emph{Step 1.} Let $\Omega\Subset M\setminus \mathcal{A}$. Since $x_{i,n}\to x^{(i)}\in \mathcal{A}$
for every $i$, there is $d_\Omega>0$ with $d_n(x)\ge d_\Omega$ on $\Omega$ for all large $n$.
Covering $\Omega$ by finitely many balls of radius $d_\Omega/C_*$ and applying \eqref{nofurther}
gives, for the exponent $p>2$ fixed in Proposition \ref{propextraction},
\begin{equation}\label{eq:global-local-Lp}
\|e^{u_n}\|_{L^p(\Omega)}\le C_\Omega .
\end{equation}

\medskip
\noindent \emph{Step 2.} We claim that, after passing to a subsequence, either
\begin{enumerate}
\item[(a)] $(v_n)$ is bounded in $C^{1,\alpha}_{loc}(M\setminus \mathcal{A})$, or
\item[(b)] $v_n\to-\infty$ locally uniformly on $M\setminus \mathcal{A}$.
\end{enumerate}
This follows exactly as in Step 3 of the proof of Proposition \ref{propclusterclearing}, considering the following:
 the bound \eqref{eq:global-local-Lp}, which by elliptic
estimates gives
$$\|z_n\|_{W^{2,p}(\Omega)}+\|z_n\|_{C^{1,\alpha}(\overline\Omega)}\le C_\Omega$$
for $z_n$ solving
$$-\Delta_g z_n=\lambda_n(he^{u_n}-1)\quad \text{in }\Omega, \qquad  z_n=0\quad \text{on }\partial\Omega;$$
the bound
\[
\int_\Omega v_n^+\,dV_g\le \int_\Omega u_n^+\,dV_g+\|u_n-v_n\|_{L^1(M)}
\le \int_M e^{u_n}\,dV_g+o(1)\le C,
\]
which by the previous line gives $\int_\Omega H_n^+\,dV_g\le C_\Omega$ for the harmonic part
$H_n:=v_n-z_n$; and the Harnack alternative for $H_n$ (\cite[Thm.\ 8.20]{gilbargtrudinger}), which is incompatible on overlapping balls and therefore propagates over the
connected set $M\setminus \mathcal{A}$.

\medskip
\noindent\emph{Step 3.} This is Step 4 of the proof of
Proposition \ref{propclusterclearing}, with the inductive hypothesis $\mathcal Q(A_j)$ replaced by
Proposition \ref{propclusterquant} applied to the root $A_\alpha:=I(\alpha)$ of $\mathcal T(a_\alpha)$, which yields
\begin{equation}\label{eq:root-atom-lower}
\mu(\{a_\alpha\})\ \ge\ 8\pi m_\alpha\ \ge\ 8\pi ,\qquad \alpha=1,\dots,N.
\end{equation}
(Indeed $B^g_{Rs_{A_\alpha,n}}(x_{A_\alpha,n})\subset B^g_\delta(a_\alpha)$ for $n$ large; let
$n\to\infty$, then $R\to\infty$, then $\delta\downarrow0$.)

Explicitly: fix $\alpha$ and $r>0$ with $\overline{B^g_{2r}(a_\alpha)}\cap \mathcal{A}=\{a_\alpha\}$.
Under (a) we have
$$\|v_n\|_{L^\infty(\partial B^g_r(a_\alpha))}\le C_r.$$ Let $\zeta_n$ solve
\[
-\Delta_g\zeta_n=\lambda_n(he^{u_n}-1)\ \text{ in }B^g_r(a_\alpha),\qquad
\zeta_n=-C_r\ \text{ on }\partial B^g_r(a_\alpha).
\]
Then $v_n-\zeta_n$ is $g$-harmonic and nonnegative on the boundary, so $v_n\ge\zeta_n$ in
$B^g_r(a_\alpha)$. Since $\mu_n$ has uniformly bounded mass, $(\zeta_n)$ is bounded in $W^{1,s}$ for
every $s<2$. By \eqref{eq:global-local-Lp} and elliptic estimates, $(\zeta_n)$ is also bounded in $W^{2,p}_{\loc}(\overline{B_r^g(a_\alpha)}\setminus\{a_\alpha\})$ for some $p>2$ and, up to a subsequence, converges a.e.\ and locally uniformly away from $a_\alpha$ to
$\zeta$ with $-\Delta_g\zeta=\mu\llcorner B^g_r(a_\alpha)-\lambda_{\infty}\,dV_g$, $\zeta=-C_r$ on the boundary.
By \eqref{eq:root-atom-lower} and
\eqref{greenest},
\begin{equation}\label{eq:global-log-lower}
\zeta(x)\ \ge\ 4m_\alpha\log\frac{1}{d_g(x,a_\alpha)}-C\qquad\text{near }a_\alpha .
\end{equation}
On each fixed annulus $A^g_{\eta,r/2}(a_\alpha)$, $0<\eta<r/2$, alternative (a) gives a uniform
$C^1$ bound on $v_n$, hence $\int_{A^g_{\eta,r/2}}e^{qv_n}\,dV_g\le C_\eta$ for some $q>1$; by
Lemma \ref{lemclose} and Lemma \ref{exppqM} (with exponent $1$),
$\int_{A^g_{\eta,r/2}}|e^{u_n}-e^{v_n}|\,dV_g\to0$. Hence, by $v_n\ge\zeta_n$, Fatou, and
$\int_M e^{u_n}\,dV_g\le(\min_M h)^{-1}$,
\[
\int_{A^g_{\eta,r/2}(a_\alpha)}e^{\zeta}\,dV_g\ \le\ \liminf_{n\to\infty}\int_{A^g_{\eta,r/2}(a_\alpha)}e^{v_n}\,dV_g\ \le\ C,
\]
with a bound independent of $\eta$. Letting $\eta\downarrow0$ and using monotone convergence
contradicts \eqref{eq:global-log-lower}, since $m_\alpha\ge1$ gives
$e^{\zeta}\ge C^{-1}d_g(\cdot,a_\alpha)^{-4m_\alpha}\notin L^1$ near $a_\alpha$.

\medskip
\noindent\emph{Step 4.} Hence (b) holds. Fix $\delta>0$ small and set
$$K_\delta:=M\setminus\bigcup_{\alpha=1}^N B^g_\delta(a_\alpha)\Subset M\setminus \mathcal{A}.$$
As in Step 5 of the proof of Proposition \ref{propclusterclearing}, (b) gives $\int_{K_\delta}e^{qv_n}\,dV_g\to0$ for every
$q>1$, and Lemma \ref{exppqM} with exponent $1$, together with Lemma \ref{lemclose}, yields
$\int_{K_\delta}|e^{u_n}-e^{v_n}|\,dV_g\to0$, whence $\int_{K_\delta}e^{u_n}\,dV_g\to0$. Since
$\lambda_n h$ is uniformly bounded, this holds for every small $\delta>0$ and \eqref{nodiffuseeq}
follows.
\end{proof}

\begin{proof}[Proof of Theorem~\ref{trmmain} (completed)]
Parts (i)-(iii) follow  from Proposition~\ref{propextraction}-(i).
It remains to prove \textup{(iv)}. 

Let $a_1,\dots,a_N$ be the distinct points among
$x^{(1)},\dots,x^{(\ell)}$, and recall that
\[
I(\alpha):=\{i:x^{(i)}=a_\alpha\},
\qquad
m_\alpha:=|I(\alpha)|.
\]
Applying Proposition~\ref{propclusterquant} to the roots  $A=I(\alpha)$, recalling that $S_{A,n}=S_\alpha>0$, we obtain
\begin{equation}\label{massmalpha}
    \lim_{R\to+\infty}\lim_{n\to\infty}
\int_{B^g_{R s_{A,n}}(x_{A,n})} \lambda_nhe^{u_n}\,\dV
=
8\pi m_\alpha,
\end{equation}
and, using that $x_{A,n}\to a_\alpha$ for any choice of $x_{A,n}$ in the cluster $I(\alpha)$, with \eqref{clusterneck} for $I(\alpha)$:
\[
\lim_{n\to\infty}
\int_{B^g_{S_\alpha/2}(a_\alpha)} \lambda_nhe^{u_n}\,\dV
=
8\pi m_\alpha.
\]
Together with Proposition~\ref{nodiffuse}, it follows that
\[
\lim_{n\to\infty}
\int_{M} \lambda_nhe^{u_n}\,\dV
=
8\pi \sum_{\alpha=1}^N m_\alpha = 8\pi\ell.
\]
Since all the mass $8\pi\ell$ is exhausted in the disjoint balls $B^g_{Rr_{i,n}}(x_{i,n})$ as $R\to+\infty$ by \eqref{convun8pi}, part (iv) of Theorem \ref{trmmain} follows.
\end{proof}

\section{An example of clustering: Proof of Theorem \ref{trmclusterexample}}\label{secexample}

In order to prove Theorem \ref{trmclusterexample}, we will first need to obtain precise profile estimates on  suitable approximate blow-up sequences.
Fix $r_0<\inj(M)/10$. Let $\chi\in C_c^\infty([0,2r_0))$ be equal to $1$ on $[0,r_0]$. For $\varepsilon>0$ and $\xi\in M$, set
\[
 \widetilde\rho_{\varepsilon,\xi}(x)
 :=\chi(d_g(x,\xi))\frac{8\varepsilon^2}{(\varepsilon^2+d_g(x,\xi)^2)^2},
 \qquad
 \rho_{\varepsilon,\xi}
 :=8\pi\frac{\widetilde\rho_{\varepsilon,\xi}}
 {\int_M\widetilde\rho_{\varepsilon,\xi}\dV}=:\lambda_{\eps,\xi}\tilde{\rho}_{\eps,\xi}.
\]
Let $P_{\varepsilon,\xi}$ be the unique solution of
\begin{equation}\label{projectedbubble}
 -\Delta_gP_{\varepsilon,\xi}=\rho_{\varepsilon,\xi}-8\pi,
 \qquad \int_MP_{\varepsilon,\xi}\dV=0.
\end{equation}

\begin{lemma}\label{lem:lambda_eps-estimate}
    There exists $C>0$ and $\eps_*>0$ such that for any $0<\eps<\eps_*$ and $\forall \xi\in M$ 
    \begin{equation}\label{eq:lambda_eps-estimate}
        \big|\lambda_{\eps,\xi}-1\big|\leq C\eps^2 \log\Big(\frac{1}{\eps}\Big).
    \end{equation}
\end{lemma}
\begin{proof}
    Recalling \eqref{liouvol} and the definition of $\tilde{\rho}_{\eps,\xi}$, it is easy to estimate
    \begin{align*}
        \Big|\int_M \tilde{\rho}_{\eps,\xi}\,dV_g -8\pi\Big|\leq C\eps^2 \log\Big(\frac{1}{\eps}\Big),
    \end{align*}
    for a suitable $C>0$ and $\eps$ small. Then \eqref{eq:lambda_eps-estimate} immediately follows.
\end{proof}

Given $x,\xi\in M$, let us define
\begin{equation}\label{eq:V_eps}
    V_{\eps,\xi}(x):=-2\chi(d_g(x,\xi))\log\big(\eps^2+d_g(x,\xi)^2\big).
\end{equation}
Recall that
\begin{align*}
    -\Delta\big(-2\log (\eps^2+|x|^2)\big)=\frac{8\eps^2}{(\eps^2+|x|^2)^2}.
\end{align*}
We also define the regular part $H$ of the Green's function as
\begin{equation}\label{eq:H-green-def}
    G(x,y):=\chi(d_g(x,y))\frac{1}{2\pi}\log\frac{1}{d_g(x,y)}+H(x,y), 
\end{equation}
recalling it fundamental role for the blow-up analysis in \cite{chenlin-cpam02}. 

We now obtain a uniform estimate of the projected bubble $P_{\varepsilon,\xi}$ in terms of the profile $V_{\epsilon,\xi}$ and the regular part of the Green's function.

\begin{lemma}\label{lem:profile-estimate}
    For all $\beta\in (0,1)$ there exists $C_\beta>0,\eps_\beta>0$ such that $\forall \eps<\eps_\beta$ and $\forall \xi\in M$ one has
    \begin{equation}\label{eq:profile-estimate}
        \sup_{x\in M}\Big| P_{\eps,\xi}(x)- V_{\eps,\xi}(x)-8\pi H_{\xi}(x)\Big|\leq C_\beta \eps^{2-\beta},
    \end{equation}
    where $H_\xi(x):=H(x,\xi)$. In particular, there exists $C>1$ such that
    \begin{equation}\label{eq:profile-est-conseq}
        \frac{1}{C(\eps^2 + d_g(x,\xi)^2)^2}\leq e^{P_{\eps,\xi}(x)}\leq\frac{C}{(\eps^2 + d_g(x,\xi)^2)^2} \qquad \forall x\in M.
    \end{equation}
\end{lemma}
\begin{proof}
The idea is to compute the Laplacian of the left-hand side in \eqref{eq:profile-estimate} and then use elliptic estimates.
To begin, define
\begin{align}\label{eq:bar-c_eps}
    \bar{c}_{\eps,\xi}:=\int_M(V_{\eps,\xi}+8\pi H_\xi)\,dV_g,
\end{align}
so that
\begin{align*}
    \int_M \big( P_{\eps,\xi}- V_{\eps,\xi}-8\pi H_\xi
+\bar{c}_{\eps,\xi}\big)\,dV_g=0.
\end{align*}
Using \eqref{eq:bar-c_eps} and the fact that $\int_M G\,dV_g=0$, one can obtain the estimate
\begin{equation}\label{eq:bar-c_eps-est}
    |\bar{c}_{\eps,\xi}|\leq C \eps^2 \log\frac{1}{\eps},
\end{equation}
for a suitable $C>0$ and $\eps$ small.

Let us next consider geodesic normal coordinates centered at $\xi$; on radial functions $f=f(t)$ one has 
\begin{align*}
    \Delta_g f=f''+\frac{1}{t} f' +\p_t(\log\sqrt{\det g}) f'=:\Delta f+\Psi_g f',
\end{align*}
where the angular term $\Psi_g(t,\theta):=\p_t(\log\sqrt{\det g})$ satisfies
\begin{equation}\label{eq:Psi_g-est}
    |{\Psi_g(t,\theta)}|\leq C_g t,
\end{equation}
    for a suitable constant $C_g>0$ independent of $\xi$.
One computes
\begin{align*}
    \Delta_g V_{\eps,\xi}(t,\theta)=&-2\chi''(t)\log(\eps^2+t^2)-8\chi'(t)\frac{t}{\eps^2+t^2}-8\chi(t)\frac{1}{\eps^2+t^2}+8\chi(t)\frac{t^2}{(\eps^2+t^2)^2} \\
    &-\frac{2}{t}\chi'(t)\log(\eps^2+t^2)-2\Psi_g(t,\theta)\Big(\chi'(t)\log(\eps^2+t^2)+\chi(t)\frac{2t}{\eps^2+t^2}\Big),
\end{align*}
while, recalling Lemma \ref{green} and \eqref{eq:H-green-def} one has
\begin{align*}
    8\pi \Delta_g H_\xi=8\pi +4\chi''(t)\log t+8\chi'(t)\frac{1}{t}+4\chi'(t)\frac{\log t}{t}+4\Psi_g(t,\theta)\Big(\chi'(t)\log t +\frac{\chi(t)}{t}\Big).
\end{align*}
Coupling these estimates with \eqref{projectedbubble}
we get
\begin{align*}
    \Delta_g\big(P_{\eps,\xi}(x)- &V_{\eps,\xi}(x)-8\pi H_\xi(x)\big)=-\lambda_{\eps,\xi}\chi\frac{8\eps^2}{(\eps^2+t^2)^2}+8\chi\Big(\frac{1}{\eps^2+t^2}-\frac{t^2}{(\eps^2+t^2)^2}\Big) \\
    &+2\chi''\big(\log(\eps^2+t^2)-2\log t\big)+8\chi'\Big(\frac{t}{\eps^2+t^2}-\frac{1}{t}\Big)+2\frac{\chi'}{t}\big(\log(\eps^2+t^2)-2\log t\big) \\
    &+2 \Psi_g \Big[\chi'\big(\log(\eps^2+t^2)-2\log t\big)+\chi\Big(\frac{2t}{\eps^2+t^2}-\frac{2}{t}\Big)\Big].
\end{align*}
We next notice that $\chi'(t),\chi''(t)\not=0$ only if $t\in(r_0,2r_0)$ ($r_0$ fixed), and in such region 
\begin{align*}
  \log(\eps^2+t^2)-\log t^2=O_{r_0}(\eps^2), \qquad \frac{t}{\eps^2+t^2}-\frac{1}{t}=O_{r_0}(\eps^2), 
\end{align*}
therefore
\begin{align*}
    \Delta_g\big(P_{\eps,\xi}(x)- &V_{\eps,\xi}(x)-8\pi H_\xi(x)\big)=-\lambda_{\eps,\xi}\chi\frac{8\eps^2}{(\eps^2+t^2)^2}+\chi\frac{8\eps^2}{(\eps^2+t^2)^2} \\
    &+ O_{r_0}(\eps^2)+2 \Psi_g \chi \frac{-2\eps^2}{t(\eps^2+t^2)}.
\end{align*}
Thus, recalling \eqref{eq:lambda_eps-estimate}, \eqref{eq:Psi_g-est}, we finally obtain the pointwise estimate
\begin{equation*}
    \Big| \Delta_g\big(P_{\eps,\xi}(x)- V_{\eps,\xi}(x)-8\pi H_\xi(x)\big)\Big|\leq C\frac{\eps^2\log(1/\eps)}{\eps^2 + d_g(x,\xi)^2}, \quad \forall x\in M.
\end{equation*}
An integration then yields the integral estimate
\begin{align*}
    \|{\Delta_g\big(P_{\eps,\xi}(x)- V_{\eps,\xi}(x)-8\pi H_\xi(x)\big)}\|_{L^p(M)}\leq C_p \eps^{\frac{2}{p}}\log\frac{1}{\eps},
\end{align*}
for any $p>1$ and a suitable  $C_p>0$, and elliptic estimates together with \eqref{eq:bar-c_eps-est} provide
\begin{align*}
    \| P_{\eps,\xi}- V_{\eps,\xi}-8\pi H_\xi + \bar{c}_{\eps,\xi}\|_{W^{2,p}}\leq C_p \eps^{\frac{2}{p}}\log\frac{1}{\eps}.
\end{align*}
The estimate \eqref{eq:profile-estimate} then follows from the Sobolev embedding $W^{2,p}\to C^{0,\alpha}$ for a suitable choice of $p$ close enough to $1$. Finally, \eqref{eq:profile-est-conseq} is an immediate consequence of \eqref{eq:profile-estimate}.    
\end{proof}

\medskip

In the sequel, we let $\varepsilon_n=o(\delta_n)$, $\delta_n\downarrow0$, and
\[
 \xi_{1,n}=\exp_{x_0}\!\left(\frac{\delta_n}{2}e_n\right),
 \qquad
 \xi_{2,n}=\exp_{x_0}\!\left(-\frac{\delta_n}{2}e_n\right),
\]
for given \emph{unitary} vectors $e_n\in T_{x_0}M$ (which will be important to choose depending on $n$).
Set $P_{i,n}:=P_{\varepsilon_n,\xi_{i,n}}$, $W_n:=P_{1,n}+P_{2,n}$, and
\[
        Z_n:=\int_Mhe^{W_n}\dV.
\]
 We also set $d_{i,n}(x):=d_g(x,\xi_{i,n})$ and
\begin{align*}
    \rho_{i,n}:=\rho_{\eps_{n},\xi_{i,n}}, \qquad \lambda_{i,n}:=\lambda_{\eps_n,\xi_{i,n}}, \qquad B_{i,n}:=B^g_{\delta_n/2}(\xi_{i,n}), \qquad \Omega_n=M\setminus \big(B_{1,n} \cup B_{2,n}\big).
\end{align*}
The $B_{i,n}$'s are the \emph{core regions} of our double-bubble $W_n$, while $\Omega_n$ are the external regions.

\begin{lemma}\label{lem:core-exp}
    For  $x\in B_{1,n}$, let
    \begin{equation}\label{eq:Phi_n-def}
        \Phi_{1,n}(x):=\log h(x) + 8\pi\big( H_{\xi_{1,n}}(x)+ H_{\xi_{2,n}}(x)\big) -2\log(\eps_n^2 + d_{2,n}(x)^2).
    \end{equation}
    Then, for $n$ large enough, one has
    \begin{align}\label{eq:core-exp}
        h \, e^{W_n}=\frac{e^{\Phi_{1,n}}}{8\eps_n^2}\tilde{\rho}_{1,n}\big(1+ O(\eps_n^{3/2})\big) \qquad \text{in $B_{1,n}$.}
    \end{align}
    Moreover,
    \begin{align}\label{eq:Phi_n-lip-est}
        \big| \Phi_{1,n}(x)-\Phi_{1,n}(\xi_{1,n})\big|\leq C\frac{d_{1,n}(x)}{\delta_n} \qquad \text{in $B_{1,n}$.}
    \end{align}
    The same result holds with the indices $1$ and $2$ interchanged.
\end{lemma}
\begin{proof}
    The estimate \eqref{eq:core-exp} follows immediately by applying the profile estimate \eqref{eq:profile-estimate} with $\beta=\frac{1}{2}$ to \emph{both} bubble profiles $P_{1,n}$ and $P_{2,n}$, yielding
    \begin{align*}
        h(x) \, e^{W_n(x)}=\frac{h(x)e^{8\pi( H_{\xi_{1,n}}(x)+ H_{\xi_{2,n}}(x))}e^{O(\eps_n^{3/2})}}{(\eps_n^2+ d_{1,n}(x)^2)^2(\eps_n^2+d_{2,n}(x)^2)^2},
    \end{align*}
    which is exactly \eqref{eq:core-exp} for $x\in B_{1,n}$.

    The second estimate follows instead from \eqref{eq:Phi_n-def} and the triangle inequality.
\end{proof}

\begin{lemma}\label{lem:outer-exp-Z-n}
    For $n$ large and a suitable $C>0$, one has
    \begin{align}\label{eq:outerint-est}
        \int_{\Omega_n}h \, e^{W_n}\,dV_g\leq\frac{C}{\delta_n^6}, \qquad \qquad \int_{\Omega_n}\rho_{i,n}\,dV_g\leq C\frac{\eps_n^2}{\delta_n^2}, \qquad i=1,2,
    \end{align}
    and 
    \begin{align}\label{eq:Z_n-est}
    \frac{1}{C \eps_n^2 \delta_n^4}\leq Z_n\leq \frac{C}{\eps_n^2 \delta_n^4}.    
    \end{align}
\end{lemma}
\begin{proof}
    All of these estimates are direct consequences of \eqref{eq:profile-est-conseq}.
\end{proof}

\begin{lemma}\label{lem:core-weights}
    Define, for $i=1,2$, the \emph{core weights}
    \begin{align}\label{eq:core-weights-def}
        \theta_{i,n}:=2\pi\frac{e^{\Phi_{i,n}(\xi_{i,n})}}{\eps_n^2 Z_n}, \qquad \eta_{i,n}:=\theta_{i,n}-1,
    \end{align}
    where $\Phi_{i,n}$ is defined in \eqref{eq:Phi_n-def}. Let us also set
    \begin{align}\label{eq:b_n-def}
        b_{i,n}:=h(\xi_{i,n}) e^{8\pi H(\xi_{i,n},\xi_{i,n})}.
    \end{align}
    Then it holds
\begin{align}\label{eq:core-weights-form}
    \frac{2}{Z_n}\int_{B_{i,n}}h \, e^{W_n}\,dV_g=\theta_{i,n}\Big( 1+ O\big(\frac{\eps_n}{\delta_n}\big)\Big).
\end{align}
    In particular,
    \begin{align}\label{eq:core-est}
        \theta_{i,n}=\frac{2 b_{i,n}}{b_{1,n}+b_{2,n}}+ O\Big(\frac{\eps_n}{\delta_n}\Big), \quad \text{and} \quad |\eta_{i,n}|\leq C\Big(\delta_n+\frac{\eps_n}{\delta_n}\Big).
    \end{align}
\end{lemma}
\begin{proof}
    By applying \eqref{eq:core-exp}, we have
\begin{align*}
    \int_{B_{1,n}}h \, e^{W_n}\,dV_g=\big(1+O(\eps_n^{3/2})\big) \frac{e^{\Phi_{1,n}(\xi_{1,n})}}{8\eps_n^2}\int_{B_{1,n}}\tilde{\rho}_{1,n} e^{(\Phi_{1,n}(x)-\Phi_{1,n}(\xi_{1,n}))}\,dV_g.
\end{align*}
Moreover, \eqref{eq:Phi_n-lip-est} together with the elementary estimate $|e^s-1|\leq |s|e^{|s|}$ imply 
\begin{align*}
    \int_{B_{1,n}}\tilde{\rho}_{1,n}e^{(\Phi_{1,n}(x)-\Phi_{1,n}(\xi_{1,n}))}\,dV_g&\leq \int_{B_{1,n}}\tilde{\rho}_{1,n}\,dV_g+O(\delta_n^{-1})\int_{B_{1,n}}\tilde{\rho}_{1,n} d_{1,n}\,dV_g\\
    &=8\pi + O\Big(\frac{\eps_n^2}{\delta_n^2}\Big)+O\Big(\frac{\eps_n}{\delta_n}\Big).
\end{align*}
Hence, substituting above, we get
\begin{align*}
    \int_{B_{1,n}}h \, e^{W_n}\,dV_g=\pi\frac{e^{\Phi_{1,n}(\xi_{1,n})}}{\eps_n^2}\big(1+O(\eps_n/\delta_n)\big),
\end{align*}
and \eqref{eq:core-weights-form} immediately follows for $i=1$. Similarly, we can also prove \eqref{eq:core-weights-form} for $i=2$.

Coupling \eqref{eq:core-weights-form} with \eqref{eq:outerint-est}, \eqref{eq:Z_n-est} and noticing that \eqref{eq:core-weights-form} implies that $\theta_{i,n}\leq C$ uniformly, we obtain
\begin{align}\label{eq:theta_i-sum}
    2=\frac{2}{Z_n}\int_M h \, e^{W_n}\,dV_g=\theta_{1,n}+\theta_{2,n}+O\Big(\frac{\eps_n}{\delta_n}\Big).
\end{align}
Next, using the definitions \eqref{eq:Phi_n-def}, \eqref{eq:b_n-def} of $\Phi_{i,n}$ and $b_{i,n}$, together with the symmetry of $H(x,y)$, we have
\begin{align}\label{eq:theta_i-ratio}
    \frac{\theta_{1,n}}{\theta_{2,n}}=e^{\big(\Phi_{1,n}(\xi_{1,n})-\Phi_{2,n}(\xi_{2,n})\big)}=\frac{h(\xi_{1,n})e^{8\pi H(\xi_{1,n},\xi_{1,n})}}{h(\xi_{2,n})e^{8\pi H(\xi_{2,n},\xi_{2,n})}}=\frac{b_{1,n}}{b_{2,n}}.
\end{align}
From \eqref{eq:theta_i-sum} and \eqref{eq:theta_i-ratio} we get
\begin{align*}
    \theta_{2,n}\Big(\frac{b_{1,n}}{b_{2,n}}+1\Big)=2+O\Big(\frac{\eps_n}{\delta_n}\Big),
\end{align*}
    namely
    \begin{align*}
        \theta_{2,n}=\frac{2 b_{2,n}}{b_{1,n}+b_{2,n}}+O\Big(\frac{\eps_n}{\delta_n}\Big),
    \end{align*}
    which is the first equation of \eqref{eq:core-est}. Finally, using \eqref{eq:b_n-def} and the fact that the function $x\in M\mapsto h(x)e^{8\pi H(x,x)}$ is globally Lipschitz, we get
    \begin{align}\label{eq:eta_n-estimate}
        |\eta_{1,n}|\leq\Big|\frac{ b_{1,n}-b_{2,n}}{b_{1,n}+b_{2,n}}\Big|+O\Big(\frac{\eps_n}{\delta_n}\Big)\leq C\Big(\delta_n+\frac{\eps_n}{\delta_n}\Big),
    \end{align}
    which is the second estimate in \eqref{eq:core-est}.
\end{proof}
From \eqref{eq:eta_n-estimate} it is clear that one can obtain the better estimate $|\eta_{i,n}|\leq C\eps_n/\delta_n$ whenever $b_{1,n}=b_{2,n}$. This will indeed be needed for the proof of Theorem \ref{trmclusterexample}.
\begin{defin}\label{defbalance}
    In the above setting, we will call a configuration $(\xi_{1,n}),(\xi_{2,n})$ \emph{balanced} if $b_{1,n}=b_{2,n}$ for all $n\in \N$. In particular, Lemma \ref{lem:core-weights} implies that, in a balanced configuration, $|\theta_{1,n}-\theta_{2,n}|=O(\eps_n/\delta_n)$.
\end{defin}
It is easy to show that balanced configurations exist at any scale and around any point:
\begin{lemma}\label{lem:balanced-config}
    Let $x_0\in M$ and $0<\delta<\mathrm{inj}(M)$. There exists a unit vector $e\in T_{x_0}M$ such that the points $\xi_i:=\mathrm{exp}_{x_0}\Big(\pm\frac{\delta}{2}e\Big)$ satisfy $h(\xi_1)e^{8\pi H(\xi_1,\xi_1)}=h(\xi_2) e^{8\pi H(\xi_2,\xi_2)}$.
\end{lemma}
\begin{proof}
    Let $e(t)$ be the unit vector of angle $t\in[0,2\pi]$ in $T_{x_0}M$, and set $\psi(x):=h(x)e^{8\pi H(x,x)}$ and $\phi(t):=\psi\big(\mathrm{exp}_{x_0}(\frac{\delta}{2}e(t))\big)-\psi\big(\mathrm{exp}_{x_0}(-\frac{\delta}{2}e(t))\big)$. Then $\phi$ is continuous and $e(t+\pi)=-e(t)$ implies $\phi(t+\pi)=-\phi(t)$. Hence $\phi$ vanishes somewhere in $[0,\pi]$.
\end{proof}

In the next lemma we estimate, for $p<2$, the $L^p$-difference between a weighted double-bubble and the Liouville nonlinearity associated to $W_n$. In the balanced case, this will turn out to be exactly the norm $\|R_n\|_{L^p(M)}$ in  Theorem \ref{trmclusterexample}.

\begin{lemma}\label{lem:twobub-L^p-est}
    Let $1\leq p<2$ and assume that $\delta_n=\eps_n^a$ with $0<a<\frac{2-p}{p}$. Then, letting $\theta_{i,n}$ be as in Lemma \ref{lem:core-weights}, one has
    \begin{align}\label{eq:twobub-L^p-est}
        \left\| \theta_{1,n}\rho_{1,n}+\theta_{2,n}\rho_{2,n}-16\pi\frac{h e^{W_n}}{Z_n}\right\|_{L^p(M)}\leq C_p\eps_n^{\frac{2-p}{p}-a},
    \end{align}
    for a suitable constant $C_p>0$. If moreover the configuration is balanced, then $\theta_{i,n}=1+O(\eps_n^{1-a})$ and
    \begin{align}\label{eq:twobub-L^p-balanced-est}
        \left\| \rho_{1,n}+\rho_{2,n}-16\pi\frac{h e^{W_n}}{Z_n}\right\|_{L^p(M)}\leq C_p\eps_n^{\frac{2-p}{p}-a}.
    \end{align}
\end{lemma}
\begin{proof}
    Define $\omega_n:=16\pi\frac{h \, e^{W_n}}{Z_n}-\theta_{1,n}\rho_{1,n}-\theta_{2,n}\rho_{2,n}$.
    Using \eqref{eq:core-exp} in $B_{1,n}$, one has
    \begin{align*}
        16\pi\frac{h(x) \, e^{W_n(x)}}{Z_n}&=2\pi\frac{e^{\Phi_{1,n}(\xi_{1,n})}}{\eps_n^2 Z_n}\tilde{\rho}_{1,n}(x)e^{\big(\Phi_{1,n}(x)-\Phi_{1,n}(\xi_{1,n})\big)}\big(1+O(\eps_n^{3/2})\big) \\
        &=\theta_{1,n}\tilde{\rho}_{1,n}(x)e^{\big(\Phi_{1,n}(x)-\Phi_{1,n}(\xi_{1,n})\big)}\big(1+O(\eps_n^{3/2})\big), \qquad \forall x\in B_{1,n}.
    \end{align*}
    Therefore, using \eqref{eq:lambda_eps-estimate} and arguing as in the proof of Lemma \ref{lem:core-weights}, we obtain the estimate
    \begin{align*}
        |\omega_n|&\leq \theta_{1,n}\tilde{\rho}_{1,n}C\big(\frac{d_{1,n}}{\delta_n}+\eps_n^{3/2}\big)+\theta_{2,n}{\rho}_{2,n} \\
        &\leq C \tilde{\rho}_{1,n}\big(\frac{d_{1,n}}{\delta_n}+\eps_n^{3/2}\big)+ C \rho_{2,n}, \qquad \text{in $B_{1,n}$.}
    \end{align*}
A similar estimate holds in $B_{2,n}$ with  indices $1$ and $2$ interchanged.
Using this, we compute
\begin{align*}
    \int_{B_{1,n}}|\omega_n|^p\,dV_g&\leq C_p\int_{B_{\delta_n/2}(0)}\frac{\eps_n^{2p}}{ (\eps_n^2+|x|^2)^{2p}}\Big(\frac{|x|^p}{\delta_n^p}+\eps_n^{\frac{3p}{2}}\Big)\,dx + C_p\int_{B_{\delta_n/2}(0)}\frac{\eps_n^{2p}}{\delta_n^{4p}}\,dx \\
    &\leq C_p \Big(\frac{\eps_n^{2p}}{\delta_n^p}\eps_n^{2-3p}+\eps_n^{{2p}+\frac{3}{2}p}\eps_n^{2-4p}+\frac{\eps_n^{2p}}{\delta_n^{4p-2}}\Big).
\end{align*}
Recalling $\delta_n=\eps_n^a$ and that $0<a<\frac{2-p}{p}$, one sees that the main term is the first one, so 
\begin{align}\label{eq:omega_n-Lpest-int}
    \int_{B_{1,n}}|\omega_n|^p\,dV_g\leq C_p \eps_n^{2-p-ap}.
\end{align}
The same estimate also holds on $B_{2,n}$.
On the other hand, in $\Omega_n=M\setminus(B_{1,n}\cup B_{2,n})$ we can use \eqref{eq:profile-est-conseq}, \eqref{eq:Z_n-est} to get
\begin{align*}
    |\omega_n|\leq\frac{C\eps_n^2\delta_n^4}{(\eps_n^2+d_{1,n}^2)^2(\eps_n^2+d_{2,n}^2)^2}+\frac{C\eps_n^2}{(\eps_n^2+d_{1,n}^2)^2}+\frac{C\eps_n^2}{(\eps_n^2+d_{2,n}^2)^2}.
\end{align*}
Using this inequality, one can estimate
\begin{align}
\notag
    \int_{\Omega_n}|\omega_n|^p\,dV_g&\leq C_p\Big[\int_{B^g_{2\delta_n}(x_0)}\frac{\eps_n^{2p}}{\delta_n^{4p}}\,dV_g + \int_{\Omega_n\setminus B^g_{2\delta_n}(x_0)}\Big|\frac{\eps_n^2 \delta_n^4}{d_g(\cdot,x_0)^8}+\frac{\eps_n^2}{d_g(\cdot,x_0)^4}\Big|^{p}\,dV_g\Big] \\
    \label{eq:omega_n-Lpest-exp}
    &\leq C_p \eps_n^{2p}\delta_n^{2-4p}=C_p \eps_n^{2(p+a(1-2p))}.
\end{align}
Being $a<1\leq p$, we easily see that $2p+2a(1-2p)>2-p-ap$, and \eqref{eq:omega_n-Lpest-int}-\eqref{eq:omega_n-Lpest-exp} imply
\begin{align}\label{eq:omega_n-Lp-est}
    \int_M |\omega_n|^p\,dV_g\leq C_p \eps_n^{2-p-ap},
\end{align}
which, recalling the definition of $\omega_n$ at the beginning of the proof, is exactly \eqref{eq:twobub-L^p-est}.

Now, if the configuration is balanced, i.e. $b_{1,n}=b_{2,n}$, then  \eqref{eq:core-est} implies that $\theta_{i,n}=1+O(\eps_n^{1-a})$ and so $|\eta_{i,n}|\leq C \eps_n^{1-a}$ for all $n\in\N$. Since
\begin{align*}
    16\pi \frac{h e^{W_n}}{Z_n}-\rho_{1,n}-\rho_{2,n}=\omega_n+\eta_{1,n}\rho_{1,n}+\eta_{2,n}\rho_{2,n},
\end{align*}
 it only remains to get an $L^p$ estimate on $\rho_{i,n}$, which follows from
\begin{align*}
    \int_M \rho_{i,n}^p\,dV_g\leq C_p\int_M \frac{\eps_n^{2p}}{(\eps_n^2+d_{i,n}^2)^{2p}}\,dV_g \leq C_p \eps_n^{2p}\int_0^1 \frac{r \,dr}{(\eps_n^2+r^2)^{2p}}\leq C_p \eps_n^{2-2p}.
\end{align*}
It follows that $\|\eta_{i,n} \rho_{i,n}\|_{L^p(M)}\leq C_p \eps_n^{1-a}\eps_n^{\frac{2-2p}{p}}=C_p\eps_n^{\frac{2-p}{p}-a}$, which, coupled with \eqref{eq:omega_n-Lp-est}, implies \eqref{eq:twobub-L^p-balanced-est}.
\end{proof}

\medskip

\begin{rmk}
    It is crucial to consider balanced sequences since we will actually need \eqref{eq:twobub-L^p-balanced-est} in the proof of Theorem \ref{trmclusterexample} below. At the same time, the balancing condition is not needed if we only care about the vanishing of the remainder term in $H^{-1}$, as it is possible to prove the following estimate on it:
     \begin{equation*}
  \left\|\rho_{\varepsilon_n,\xi_{1,n}}+\rho_{\varepsilon_n,\xi_{2,n}}
  -16\pi\frac{he^{W_n}}{Z_n}\right\|_{H^{-1}(M)}
  \le C\left(\delta_n+\frac{\varepsilon_n}{\delta_n}\right)
  \sqrt{\log\frac1{\varepsilon_n}}.
 \end{equation*}
\end{rmk}

\medskip

\begin{proof}[Proof of Theorem \ref{trmclusterexample}]
Fix $a\in(0,\frac{2-p}{p})$ as in Lemma \ref{lem:twobub-L^p-est} and choose
\[
        \eps_n=e^{-n}, \qquad \delta_n=\eps_n^a=e^{-an}.
\]
By Lemma \ref{lem:balanced-config} applied at $x_0$ and with $\delta=\delta_n$, we find unit vectors $e_n\in T_{x_0}M$ such that the points
\begin{align*}
    \xi_{1,n}=\mathrm{exp}_{x_0}\Big(\frac{\delta_n}{2}e_n\Big), \qquad  \xi_{2,n}=\mathrm{exp}_{x_0}\Big(-\frac{\delta_n}{2}e_n\Big)
\end{align*}
satisfy $b_{1,n}=b_{2,n}$ for all $n\in \N$, where $b_{i,n}$ are given by \eqref{eq:b_n-def}.
Define $W_n=P_{1,n}+P_{2,n}$ as above and set
\[
        u_n=W_n-\log Z_n.
\]
Then $\int_Mhe^{u_n}\dV=1$ and, by~\eqref{projectedbubble},
\[
 -\Delta_gu_n
 =\rho_{\varepsilon_n,\xi_{1,n}}+\rho_{\varepsilon_n,\xi_{2,n}}-16\pi.
\]
Hence
\[
 R_n=\rho_{\varepsilon_n,\xi_{1,n}}+\rho_{\varepsilon_n,\xi_{2,n}}
 -16\pi he^{u_n},
\]
which converges to zero in $L^p(M)$ by estimate \eqref{eq:twobub-L^p-balanced-est} of Lemma \ref{lem:twobub-L^p-est}.
The local expansion of Lemma \ref{lem:profile-estimate} shows that the scale $\varepsilon_n$ around each $\xi_{i,n}$ produces a standard Liouville profile of mass $8\pi$. Moreover
\[
 \frac{d_g(\xi_{1,n},\xi_{2,n})}{\varepsilon_n}
 =\frac{\delta_n}{\varepsilon_n}=\eps_n^{a-1}\to\infty,
 \qquad
 \xi_{1,n},\xi_{2,n}\to x_0.
\]
Thus the profiles are separated at their own scales but cluster at $x_0$. Since the total mass is $16\pi$, Theorem~\ref{trmmain} shows that no further profile occurs.

Finally, the convergence of $R_n$ to zero in $H^{-1}$ follows from the embedding of $L^p$ into $H^{-1}$ for all $p>1$.
\end{proof}

\begin{rmk}
The equations \eqref{PSbar} and \eqref{PSeq} only correspond to the vanishing of $J'_{\lambda_n}(u_n)$ as $n\to +\infty$. For the sequence $(u_n)$ constructed in Theorem \ref{trmclusterexample}, we have
\begin{align*}
J_{16\pi}(u_n)=32\pi \log \delta_n+O(1)=-32\pi a n+O(1)\to -\infty.
\end{align*}
In particular, here the divergence of $J_{16\pi}$ is caused by the clustering, as for two bubbles at fixed positive distance we would get instead $J_{16\pi}(u_n)=O(1)$.
It is an open question whether adding the condition $J_{\lambda_n}(u_n)\to c\in \R$ is sufficient to prevent clustering.
\end{rmk}

\begin{rmk}
   A natural question is whether arbitrary bubbling trees can occur. More precisely, whether, after a suitable generalization of the balancing condition of Definition \ref{defbalance}, every finite tree $\mathcal{T}$ can be realized by a blowing-up Palais--Smale sequence $(u_n)$.
\end{rmk}

\section{Bubble separation: Proof of Proposition \ref{propisolatedstrong}}\label{sec:isolation}


Assume that at least two sequences $(x_{i,n})$ and $(x_{j,n})$  have the same limit $a$. Let $\delta_n$ be the smallest distance between two such centers, and relabel the bubbles so that $i=1$, $j=2$ and $d_g(x_{1,n},x_{2,n})=\delta_n$, and set
\[
 J:=\{j\in\{1,\dots,\ell\}:d_g(x_{j,n},x_{1,n})=O(\delta_n)\}.
\]
Fix
an isothermal chart $\psi:U\to\Omega\subset\mathbb R^2$ around $a$, with $\psi(a)=0$ and $g=e^{2\varphi}|dx|^2$.
Up to replacing $\delta_n$ by an asymptotically equivalent sequence,
we may assume that
\[
 \delta_n=|\xi_{1,n}-\xi_{2,n}|,
 \qquad
 \xi_{j,n}:=\psi(x_{j,n}).
\]
Set
\[
 y_{j,n}:=\frac{\xi_{j,n}-\xi_{1,n}}{\delta_n},
 \qquad j\in J.
\]
After passing to a subsequence, $y_{j,n}\to y_j$, where the points
$y_j$ are pairwise distinct.

Define
\[
 F_n(y):=\xi_{1,n}+\delta_n y,
\]
and
\[
 \widehat u_n(y)
 :=
 u_n(\psi^{-1}(F_n(y)))+2\log\delta_n,
 \qquad
 \widehat v_n(y)
 :=
 v_n(\psi^{-1}(F_n(y)))+2\log\delta_n.
\]
Then
\begin{equation}\label{eq:f_n-clust-def}
 -\Delta\widehat v_n
 =
 \lambda_n \alpha_n e^{\widehat u_n}
 -\delta_n^2\lambda_n \beta_n
 =:f_n,
\end{equation}
where
\[
 \alpha_n(y)
 :=
 e^{2\varphi(F_n(y))}
 h(\psi^{-1}(F_n(y))),
 \qquad
 \beta_n(y):=e^{2\varphi(F_n(y))}.
\]
In particular, on every fixed bounded set,
\[
 \alpha_n\to \alpha_0:=e^{2\varphi(0)}h(a)>0,
 \qquad
 \|\nabla \alpha_n\|_{L^\infty}
 +\|\nabla \beta_n\|_{L^\infty}
 =O(\delta_n).
\]
By Proposition \ref{propclusterquant},
\begin{equation}\label{convehatu}
 \lambda_n \alpha_n e^{\widehat u_n}\,dy
 \rightharpoonup
 8\pi\sum_{j\in J}\delta_{y_j}.
\end{equation}
We next identify the limit of $\widehat v_n$ away from the points
$y_j$. Fix
\[
 y_*\in\mathbb R^2\setminus\{y_j:j\in J\}
\]
and set $c_n:=\widehat v_n(y_*)$. If
\[
 K\Subset\mathbb R^2\setminus\{y_j:j\in J\},
\]
then Proposition \ref{propgradient} gives
\begin{equation}\label{125}
 \|\nabla\widehat v_n\|_{L^\infty(K)}\le C_K.
\end{equation}
Moreover, the $L^q$ estimate in Proposition~\ref{propextraction} gives, for all fixed $q>2$, a uniform
$L^q(K)$ bound for $\alpha_ne^{\widehat u_n}$, whereas \eqref{convehatu} gives $\|\alpha_ne^{\widehat u_n}\|_{L^1(K)}\to0$.
Interpolation therefore yields
\begin{equation}\label{eq:u_nhat-Ls-conv}
 \alpha_ne^{\widehat u_n}\to0
 \quad\text{in }L^s(K)
\end{equation}
for every $1<s<q$. Choosing $s>2$, local elliptic estimates, \eqref{125},
and the normalization at $y_*$ imply, after extraction,
\begin{equation}\label{126}
 \widehat v_n-c_n\longrightarrow P
 \quad\text{in }C^{1,\gamma}_{\mathrm{loc}}
 \bigl(\mathbb R^2\setminus\{y_j:j\in J\}\bigr)
\end{equation}
for some $\gamma\in(0,1)$.

We now identify $P$. By the Green representation formula,
\begin{equation}\label{127}
 \nabla\widehat v_n(y)
 =
 \delta_n
 \int_M
 \nabla_xG\bigl(\psi^{-1}(F_n(y)),z\bigr)
 \lambda_nh(z)e^{u_n(z)}\,dV_g(z)
 +O(\delta_n).
\end{equation}
Using
\[
 \nabla_xG(x,z)
 =
 -\frac1{2\pi}\frac{x-z}{|x-z|^2}+O(1),\quad \text{as }x\to z,
\]
in the chosen coordinates, together with Proposition \ref{propclusterquant}, we obtain
\begin{equation}\label{128}
 \nabla P(y)
 =
 -4\sum_{j\in J}
 \frac{y-y_j}{|y-y_j|^2}.
\end{equation}
Indeed, let $\widehat\mu_n:=\lambda_n\alpha_ne^{\widehat u_n}\,dy$, so that $\widehat\mu_n\rightharpoonup8\pi\sum_{j\in J}\delta_{y_j}$ by \eqref{convehatu}. Since the
regular part of $G$ is bounded and the centers outside $J$ are at distance $\gg\delta_n$
from $x_{1,n}$, \eqref{127} gives, for every fixed
$y\in\R^2\setminus\{y_j:j\in J\}$,
\begin{equation*}
 \nabla\widehat v_n(y)
 =-\frac1{2\pi}\int_{\R^2}\frac{y-\eta}{|y-\eta|^2}\,d\widehat\mu_n(\eta)+O(\delta_n).
\end{equation*}
Fix $0<\varepsilon<\tfrac12\min_{j\in J}|y-y_j|$. 
On $B_\varepsilon(y)$, H\"older's inequality and
\eqref{eq:u_nhat-Ls-conv} give, with $s'=\frac{s}{s-1}<2$,
\begin{equation*}
 \int_{B_\varepsilon(y)}\frac{d\widehat\mu_n(\eta)}{|y-\eta|}
 \le\|\lambda_n\alpha_ne^{\widehat u_n}\|_{L^s(B_\varepsilon(y))}
 \bigl\||y-\cdot|^{-1}\bigr\|_{L^{s'}(B_\varepsilon(y))}=o(1),
\end{equation*}
so when we send $n\to+\infty$ and use \eqref{convehatu} we readily obtain \eqref{128}.
 As a consequence,
\begin{equation}\label{129}
 P(y)
 =
 4\sum_{j\in J}\log\frac1{|y-y_j|}+C.
\end{equation}
Fix now $j\in J$ and choose
$r>0$ so small that
\[
 B_{4r}(y_j)\cap\{y_k:k\in J,\ k\ne j\}=\emptyset.
\]
By \eqref{convehatu},
\[
 \int_{A_{r,2r}(y_j)}
 \lambda_n \alpha_ne^{\widehat u_n}\,dy\to0.
\]
Hence, applying Fubini's theorem, there exists
$\rho_n\in(r,2r)$ such that
\begin{equation}\label{130}
 \int_{\partial B_{\rho_n}(y_j)}
 \lambda_n\alpha_ne^{\widehat u_n}\,d\sigma\to0.
\end{equation}
After passing to a subsequence, $\rho_n\to\rho\in[r,2r]$.

Since $e^{u_n}\in L^q(M)$ $\forall q<+\infty$ by the Moser-Trudinger inequality, then $v_n\in W^{2,q}(M)$, which, together with \eqref{strongresidual}, imply $u_n\in L^{\infty}(M)$. Then \eqref{eq:f_n-clust-def} implies $\widehat{v}_n\in W^{2,q}_{\mathrm{loc}}(\R^2)$, therefore,
for a constant vector $X\in\mathbb R^2$, we may multiply \eqref{eq:f_n-clust-def} by $X\cdot \nabla \widehat{v}_n$ and integrate by parts (the LHS) on $B_{\rho_n}(y_j)$ to obtain the following Pohozaev-type identity:
\begin{equation}\label{131}
\int_{\partial B_{\rho_n}(y_j)}
 \left[
 \frac12|\nabla\widehat v_n|^2(X\cdot\nu)
 -(X\cdot\nabla\widehat v_n)
 \partial_\nu\widehat v_n
 \right]d\sigma                                     
 =
 \int_{B_{\rho_n}(y_j)}
 f_n\,X\cdot\nabla\widehat v_n\,dy.
\end{equation}
Since $w_n=u_n-v_n$, $\delta_n\to 0$ and \eqref{strongresidual} holds,
\begin{equation}\label{132}
 \|\nabla(\widehat u_n-\widehat v_n)\|_{L^\infty_{\mathrm{loc}}(\R^2)}
 \le C\delta_n\|\nabla w_n\|_{L^\infty(M)}\leq C\bar{C}\delta_n
 =o(1).
\end{equation}
The measures $|f_n|\,dy$ have uniformly bounded mass on fixed
balls; hence \eqref{132} allows us to replace
$\nabla\widehat v_n$ by $\nabla\widehat u_n$ in the right-hand side
of \eqref{131}, at the expense of an $o(1)$ error.

For the nonlinear term, writing $f_n$ as in \eqref{eq:f_n-clust-def}, integration by parts gives
\[
\begin{split}
 \int_{B_{\rho_n}(y_j)}
 \lambda_n\alpha_ne^{\widehat u_n}
 X\cdot\nabla\widehat u_n\,dy
 &=
 \int_{\partial B_{\rho_n}(y_j)}
 \lambda_n\alpha_ne^{\widehat u_n}X\cdot\nu\,d\sigma       \\
 &\quad
 -\int_{B_{\rho_n}(y_j)}
 \lambda_ne^{\widehat u_n}X\cdot\nabla \alpha_n\,dy.
\end{split}
\]
The first term tends to zero by \eqref{130}, while the second one is
$O(\delta_n)$, since $\|\nabla \alpha_n\|_{L^\infty}=O(\delta_n)$ and
the conformal volume is uniformly bounded.

Finally, Proposition \ref{propgradient} and \eqref{132} imply
\[
 \int_{B_{2r}(y_j)}|\nabla\widehat u_n|\,dy\le C_r.
\]
Therefore the contribution of
$-\delta_n^2\lambda_n\beta_n$ is $O(\delta_n^2)$. We conclude from \eqref{126},
\eqref{131} that
\begin{equation}\label{133}
 \int_{\partial B_\rho(y_j)}
 \left[
 \frac12|\nabla P|^2(X\cdot\nu)
 -(X\cdot\nabla P)\partial_\nu P
 \right]d\sigma=0.
\end{equation}

Let
\[
 P_j(y):=
 P(y)-4\log\frac1{|y-y_j|}.
\]
A direct expansion of the left-hand side of \eqref{133} as $\rho\downarrow0$ (the choice of $r$ above was arbitrary, so we can apply the same argument to a sequence $r_k\to 0$, which will produce radii $\rho_{k}\to0$ as above)
gives
\[
 8\pi\,X\cdot\nabla P_j(y_j)=0.
\]
Since $X$ is arbitrary,
\[
 \nabla P_j(y_j)=0.
\]
Using \eqref{129}, this is equivalent to
\[
 \sum_{\substack{k\in J\\k\ne j}}
 \frac{y_j-y_k}{|y_j-y_k|^2}=0,
 \qquad j\in J.
\]
Thus $(y_j)_{j\in J}$ would be a critical point of the logarithmic
interaction functional
\[
 \Xi((y_j)_{j\in J}):=\sum_{\substack{j,k\in J\\j<k}}
 \log\frac1{|y_j-y_k|},
\]
contrary to Lemma~\ref{lem:nocriticallog}. Hence the limiting
centers are pairwise distinct. \hfill $\square$

\medskip

The proof of the following lemma can be found in \cite[Lemma A.3]{MalchiodiMartinazziThizy2024}.

\begin{lemma}
\label{lem:nocriticallog}
Let $N\ge2$ and let $\beta_{ij}>0$ for $1\le i<j\le N$. Then
\[
 \Xi(z_1,\ldots,z_N)
 :=
 \sum_{1\le i<j\le N}
 \beta_{ij}\log\frac1{|z_i-z_j|}
\]
has no critical points on
\[
 (\mathbb R^2)^N\setminus
 \{(z_1,\ldots,z_N):z_i=z_j\text{ for some }i\ne j\}.
\]
\end{lemma}


\printbibliography[heading=bibintoc]

\Addresses

\end{document}